\documentclass[12pt,reqno]{amsart}

\calclayout

\usepackage{amsmath,amsthm,amssymb}
\usepackage{mathrsfs}

\usepackage{colonequals}
\usepackage{enumitem}
\usepackage{graphicx}
\usepackage[justification=centering]{caption}

\usepackage[colorlinks=true,allcolors=blue,pdftitle={GL with pinning in 3D}]{hyperref}

\usepackage[initials,nobysame,alphabetic,msc-links,backrefs]{amsrefs}
\usepackage{color}

\newtheorem{theorem}{Theorem}[section]
\newtheorem{lemma}{Lemma}[section]
\newtheorem{lemmaA}{Lemma}[section]

\newtheorem{prop}{Proposition}[section]
\newtheorem{propA}{Proposition}[section]

\newtheorem{define}{Definition}[section]
\newtheorem{corollary}{Corollary}[section]
\newtheorem{remark}{Remark}[section]

\newtheorem*{theorem*}{Theorem}
\newtheorem*{lemma*}{Lemma}
\newtheorem*{prop*}{Proposition}
\newtheorem*{define*}{Definition}
\newtheorem*{corollary*}{Corollary}

\numberwithin{equation}{section}

\DeclareMathOperator{\diver}{div}
\DeclareMathOperator{\curl}{curl}

\newcommand{\R}{\ensuremath{\mathbb{R}}}

\newcommand{\C}{\ensuremath{\mathbb{C}}}
\newcommand{\N}{\ensuremath{\mathbb{N}}}
\newcommand{\ip}[2]{\langle#1\hspace*{.5mm},#2\rangle}
\newcommand{\norm}[3][]{#1\left\|#2#1\right\|_{#3}}
\newcommand{\ep}{\varepsilon}
\newcommand{\ex}{\mathrm{ex}}
\newcommand{\loc}{\mathrm{loc}}

\renewcommand{\u}{\mathbf{u}}
\newcommand{\A}{\mathbf{A}}

\newcommand{\pmin}{\ensuremath{\rho_\varepsilon}}

\newcommand{\fen}{\ensuremath{F_{\varepsilon, \pmin}}}
\newcommand{\critfield}{\ensuremath{H^{\varepsilon}_{c_1}}}
\newcommand{\lsol}{\ensuremath{L_{\sigma}}}

\newcommand{\hsolhom}{\ensuremath{\dot{H}_{\sigma}}}

\newcommand{\hsolt}{\ensuremath{H_{\sigma,T}}}

\newcommand{\phipmin}[1]{\phi_{#1,\pmin^2}}
\newcommand{\bpmin}[1]{B_{#1,\pmin^2}}

\newcommand{\erho}[1]{E_{\ep,\pmin}(u,#1)}

\newcommand{\meisconf}{\ensuremath{\pmin e^{i h_\ex \phi_\ep^0}, h_\ex A^0_\ep}}

\newcommand{\err}{\mathfrak{r}}
\newcommand{\minsp}{H^1(\Omega,\C) \times [A_{\ex} + H_{\mathrm{curl}}(\R^3,\R^3) ]}

\definecolor{purple}{rgb}{0.63, 0.36, 0.94}

\newcommand{\CC}{\mathscr C}
\newcommand{\de}{\delta}
\newcommand{\RR}{\mathrm{R}}
\newcommand{\GG}{\mathfrak G}
\newcommand{\RRR}{\mathfrak R}

\renewcommand{\H}{\mathcal H}
\renewcommand{\d}{\mathfrak d}

\newcommand{\een}{\ensuremath{E_{\varepsilon, \pmin}}}
\newcommand{\eed}{\ensuremath{e_{\varepsilon, \pmin}}}

\newcommand{\dist}{\mathrm{dist}}

\newcommand{\NN}{\mathcal N(\Omega)}
\newcommand{\ga}{\Gamma}

\title[First critical field for 3D pinned Ginzburg--Landau]{First critical field in the pinned three dimensional Ginzburg--Landau model of superconductivity}

\author{Mat\'{i}as D\'{i}az-Vera}
\address{Facultad de Matem\'aticas, Pontificia Universidad Cat\'olica de Chile, Vicu\~na Mackenna 4860, 7820436 Macul, Santiago, Chile}
\email{midiaz8@uc.cl}

\author{Carlos Rom\'{a}n}
\address{Facultad de Matem\'aticas e Instituto de Ingenier\'ia Matem\'atica y Computacional, Pontificia Universidad Cat\'olica de Chile, Vicu\~na Mackenna 4860, 7820436 Macul, Santiago, Chile}
\email{carlos.roman@uc.cl}

\date{\today}
\thanks{Funding information: ANID FONDECYT 1231593.}
\begin{document}
\begin{abstract} 
We study extreme type-II superconductors described by the three dimensional magnetic Ginzburg--Landau functional incorporating a pinning term $a_\ep(x)$, which we assume to be a bounded measurable function satisfying $b\leq a_\ep(x)\leq 1$ for some constant $b>0$. A hallmark of such materials is the formation of vortex filaments, which emerge when the applied magnetic field exceeds the first critical field $H_{c_1}$. In this work, we provide a lower bound for this critical field and provide a characterization of the Meissner solution, that is, the unique vortexless configuration that globally minimizes the energy below $H_{c_1}$. Moreover, we show that the onset of vorticity is intrinsically linked to a weighted variant of the \emph{isoflux} problem studied in \cites{Rom-CMP,roman-sandier-serfaty}. A crucial role is played by the $\ep$-level tools developed in \cite{roman-vortex-construction}, which we adapt to the weighted Ginzburg--Landau framework.
\end{abstract}
\maketitle

\section{Introduction}
	\subsection{The problem and brief state of the art}
	Superconductors are specific metals and alloys that, when cooled below a critical temperature (usually extremely low), lose all electrical resistance, enabling currents to flow indefinitely without energy dissipation. This phenomenon was first observed by H. Kamerlingh Onnes in 1911. To describe it, Ginzburg and Landau developed the Ginzburg--Landau model of superconductivity in 1950 \cite{GinLan}, which has proven highly effective in predicting the behavior of superconductors. This model was later shown to be a limiting case of the Bardeen--Cooper--Schrieffer (BCS) theory of superconductivity (see \cites{BCS,FHSS}). The Ginzburg--Landau model is a fundamental contribution to the field of physics, the development of which led to the awarding of several Nobel Prizes.

    \medskip
	In this article we will be interested in type-II superconductors, in which regions of normal and superconducting phases can coexist within the material. A hallmark of type-II superconductivity is the formation of quantized vortex filaments --- similar in nature to the singularities found in fluid dynamics --- that occur when the material is exposed to an external magnetic field. The magnetic field penetrates the superconductor via these vortices, and superconductivity is locally suppressed.
	
	These vortices are not static; they can move due to internal interactions or external forces. Their motion induces an electric field, which leads to energy dissipation and introduces electrical resistance, which in turn effectively degrades the superconducting state. To mitigate this, inhomogeneities may be introduced into the material to act as pinning sites, anchoring the vortices and limiting their mobility. For a more detailed discussion of the physics underlying vortex pinning and its implications, see \cites{chapman-richardson, chapman-du-gunzburguer-variable-thickness, chapman-richardson-pinning,ding-du} and the references therein.

    \medskip
	The Ginzburg--Landau functional with pinning, which models the state of an inhomogeneous superconducting sample in an applied magnetic field, assuming that the temperature is fixed and below the critical one, is
	\begin{equation}\label{GLenergy}
		GL_\ep(u,A)=\frac12\int_\Omega |\nabla_A u|^2+\frac{1}{2\ep^2}(a_\ep(x)-|u|^2)^2+\frac12\int_{\R^3}|H-H_{\ex}|^2.
	\end{equation}
	Here
	\begin{itemize}
		\item $\Omega$ is a bounded domain of $\R^3$, that we assume to be simply connected with $C^2$ boundary.
		\item $u:\Omega\rightarrow \mathbb{C}$ is the \emph{order parameter}. Its squared modulus (the density of Cooper pairs of superconducting electrons in the BCS quantum theory) indicates the local state of the superconductor: where $|u|^2\approx 1$ the material is in the superconducting phase, where $|u|^2\approx 0$ in the normal phase.
		\item $A:\R^3\rightarrow \R^3$ is the electromagnetic vector potential of the induced magnetic field $H=\curl A$.
		\item $\nabla_A$ denotes the covariant gradient $\nabla-iA$.
		\item $H_{\ex}:\R^3\rightarrow \R^3$ is a given external (or applied) magnetic field. We will assume that $H_\ex=h_\ex H_{0,\ex}$, where $H_{0,\ex}$ is a fixed vector field and $h_\ex$ is a real parameter that can be tuned, which represents the intensity of the external field. 
		\item $\ep>0$ is the inverse of the \emph{Ginzburg--Landau parameter} usually denoted $\kappa$, a non-dimensional parameter depending only on the material. We will be interested in the regime of small $\ep$, corresponding to extreme type-II superconductors. 
		\item $a_{\varepsilon}$ is a function that accounts for inhomogeneities in the material. We will assume that $a_\ep \in L^\infty(\Omega)$ and that it takes values in $[b,1]$, where $b\in (0,1)$ is a constant independent of $\varepsilon$. The regions where $a_\varepsilon = 1$ correspond to sites without inhomogeneities (we also say that there is no pinning in these regions). 
	\end{itemize}
	
	The Ginzburg--Landau model is known to be a $\mathbb U(1)$-gauge theory, which means that all the meaningful physical quantities, including the free-energy within $\Omega$
	$$
	\fen(u,A)\colonequals \frac12\int_\Omega |\nabla_A u|^2+\frac{1}{2\ep^2}(a_\ep(x)-|u|^2)^2+|\curl A|^2,
	$$
	are invariant under the gauge transformations 
	$$
	u\mapsto ue^{i\phi}\quad\mathrm{and}\quad A\mapsto A+\nabla \phi,
	$$
	where $\phi$ is any sufficiently regular real-valued function.
	An extremely important quantity in the analysis of the model is the \emph{vorticity}, which is defined, for any sufficiently regular configuration $(u,A)$, as
	$$
	\mu(u,A)=\curl (iu,\nabla_A u)+\curl A,
	$$
	where $(\cdot,\cdot)$ denotes the scalar product in $\C$ identified with $\R^2$. It corresponds to the gauge-invariant version of twice the \emph{Jacobian} of $u$.
	As $\ep\to 0$, the vorticity essentially concentrates in a sum of quantized Dirac masses supported on co-dimension 2 objects. This together with the crucial fact that most of the energy concentrates around these objects is what has essentially allowed mathematicians to analyze the model. 
	
	\medskip 
	The 2D literature on the Ginzburg–Landau model with pinning is extensive. For a detailed overview of the mathematical aspects of this subject, we refer the interested reader to our recent work \cite{diaz-roman-2d} and the references therein. 
	
	In contrast, the 3D problem remains mostly open. Dos Santos \cite{DosSantos3D} considered the problem without external magnetic field and without gauge, that is, when $A=0$ and $H_\ex=0$. By taking a pinning term (independent of $\ep$) which takes some value $b\in(0,1)$ in a non-empty open and strictly convex set compactly contained in $\Omega$, and equal to $1$ otherwise, he was able to identify the limiting one dimensional singular set of minimizers of the energy, under the prescription of a Dirichlet boundary condition with non-zero degree. This result extends the classical findings \cites{Riv,sandier-surface-ball-construction} for the homogeneous Ginzburg–-Landau model, where the limiting singular set is a mass-minimizing current that corresponds to a minimal connection, a concept introduced in \cite{BreCorLie}. The pinning modifies the structure of the limiting singular set, which, as identified in \cite{DosSantos3D}, forms a geodesic link.

    \medskip
	Recently, Dos Santos, Rodiac, and Sandier \cite{DosRodSan} considered the full Ginzburg–Landau energy with a rapidly oscillating pinning term. They studied both the periodic and random stationary cases, where the oscillations are much faster than $\ep$. Although most of their results concerned the 2D situation, under the assumption that the average (in the periodic case) or the expectation (in the random case) are equal to $1$, in 3D they found that in the limit $\ep \to 0$ the limiting functional is the same as that found by \cite{BalJerOrlSon1} in the homogeneous case, showing that in this special situation one can reduce the study of the problem to the case where the pinning term is replaced by its average or expectation,  thus yielding the homogeneous limiting energy.
	
	In the homogeneous case, when $a_\ep\equiv 1$ in $\Omega$, there has been considerable development in the analysis of the model, especially in the last two decades. In particular, Baldo, Jerrard, Orlandi, and Soner \cites{BalJerOrlSon1, BalJerOrlSon2} described the asymptotic behavior of the energy functional by deriving a mean-field model as $\ep\to 0$. On the other hand, in \cite{roman-vortex-construction}, mathematical tools were developed to study the vortex filaments at the $\ep$-level, which has enabled a detailed description of the vortices when the intensity of the applied field is close to the so-called \emph{first critical field} $H_{c_1}$; see \cites{Rom-CMP,roman-sandier-serfaty, RSS2}.
	
	Below this value, which is of order $O(|\log\ep|)$, as first established by Abrikosov \cite{Abr}, the superconductor remains fully in its superconducting phase, meaning $|u|$ uniformly close to 1, and the applied field is expelled from the material. However, when the field intensity exceeds $H_{c_1}$, vortex filaments appear within the material. The order parameter $u$ vanishes in the center of each vortex core, and the applied field penetrates the material through the vortices.

	Mathematically, the leading order value of the first critical field for the homogeneous functional in 3D was first derived in \cites{BalJerOrlSon2}. More precisely, they proved that the vorticity of a minimizing sequence $(u_\ep,A_\ep)$ of $GL_\ep$ (when $a_\ep\equiv 1$ in $\Omega$), up to extraction, is such that
	\begin{equation}\label{firstweak}
		\frac{\mu(u_\ep,A_\ep)}{|\log\ep|}\ \ \overset{\ep \to 0}\longrightarrow \ \ 
		\left\{
		\begin{array}{cl}
			0&\mathrm{if}\ \lim\limits_{\ep\to0}\dfrac{h_\ex}{|\log\ep|}<\dfrac1{2\RR_0}\\
			\mbox{a non zero limit}&\mathrm{if}\ \lim\limits_{\ep\to0}\dfrac{h_\ex}{|\log\ep|}>\dfrac1{2\RR_0}
		\end{array}
		\right.
	\end{equation}
	in a certain weak topology, where the constant $\RR_0$ is described later. This result gives $H_{c_1}$ up to an error $o(|\log \ep|)$ as $\ep \to 0$ and agrees with the previous work by Alama, Bronsard, and Montero \cite{AlaBroMon} in the special case where $\Omega$ is a ball.
	
	This result was improved in \cite{Rom-CMP}, where it is shown that there exist constants $K_0,K_1>0$ such that 
	\begin{equation*}
		H_{c_1}^0-K_0\log |\log \ep|\leq H_{c_1}\leq H_{c_1}^0+K_1.
	\end{equation*}
	More significantly, given a global minimizer $(u_\ep,A_\ep)$ of $GL_\ep$: 
	\begin{itemize}
		\item if $h_\ex\leq H_{c_1}^0-K_0\log |\log \ep|$, then $\|u_\ep\|_\infty \to 1$ as $\ep\to 0$ and therefore $u_\ep$ does not have vortices;
		\item if $h_\ex\geq H_{c_1}^0+K_1$, then $u_\ep$ does have vortices. 
	\end{itemize}
	In \cite{roman-sandier-serfaty}, the lower bound was improved, replacing $-K_0\log |\log \ep|$ by a constant $-K_0$ in a special situation. This improvement is heavily based on the boundedness of the vorticity of the minimizers of $GL_\ep$ slightly above the first critical field, which is a main result in \cite{roman-sandier-serfaty}.
	
	\medskip
	To our knowledge, there are no results in the mathematics literature concerning the first critical field for the pinned Ginzburg--Landau functional in 3D, which is the main motivation of our article.

	\subsection{Approximation of the Meissner state}
	In this paper, we provide a precise approximation of the \emph{Meissner state}, that is, the unique (modulo gauge-invariance) solution of the Ginzburg--Landau equations without vortex filaments, which, in turn, allows us for providing a lower bound for the first critical field, which is of order $O(|\log\ep|)$ as in the homogeneous case.
	Our approximating configuration is given by
	\begin{equation}\label{meissnerconf}
		(\meisconf),
	\end{equation}
	where
	\begin{itemize}
		\item $\pmin$ is the unique positive real-valued minimizer in $H^1(\Omega)$ of the energy functional without magnetic components
		\begin{equation}\label{energywithoutmagneticterm}
			E_\ep(u) \colonequals \frac{1}{2}\int_\Omega |\nabla u|^2 + \frac{1}{2\ep^2}(a_\ep-|u|^2)^2.
		\end{equation}
		The Euler--Lagrange equation associated to this functional is
		\begin{equation}\label{pde rho}
			\left\{
			\begin{array}{rcll}
				-\Delta \pmin &=& \dfrac{\pmin}{\varepsilon^2}(a_\varepsilon-\pmin^2) &\mathrm{in}\ \Omega\\
				\dfrac{\partial \pmin}{\partial \nu} &=& 0&\mathrm{on}\ \partial \Omega.
			\end{array}
			\right.
		\end{equation}
		Observe that, since $a_\ep$ takes values in $[b,1]$, the maximum principle yields the following uniform bounds for $\pmin$.
        \begin{equation}\label{rho uniform bound}
            b \leq \pmin^2 \leq 1.
        \end{equation}
		
		\item $A^0_\ep$ is the unique minimizer in $[A_{0,\ex} + \hsolhom^1(\R^3,\R^3)]$ of the energy functional
		\begin{equation}\label{reducedenergy}
			J_\ep(A) \colonequals \frac{1}{2}\int_\Omega \pmin^2|A-\nabla \phipmin{A}|^2 + \frac{1}{2}\int_{\R^3} |\curl(A-A_{0,\ex})|^2,
		\end{equation}
		where $\hsolhom^1(\R^3,\R^3)$ is the space of divergence-free vector fields in the homogeneous Sobolev space $\Dot{H}^1(\R^3,\R^3)$ and $\phipmin{A}$ is the unique solution in $H^1(\Omega)$ with zero average, i.e. $\int_\Omega \phipmin{A} = 0$, of the elliptic equation
		\begin{equation}\label{phi rho equation}
			\left\{
			\begin{array}{rcll}
				\diver\left( \pmin^2(A - \nabla \phi_{A,\pmin^2})\right)  &=& 0 &\mathrm{in}\ \Omega\\
				\pmin^2(A-\nabla \phipmin{A}) \cdot \nu &=& 0 &\mathrm{on}\ \partial \Omega.
			\end{array}
			\right.
		\end{equation}
		The structure of the equation implies that 
		\begin{equation}\label{b0 existence}
			A-\nabla \phipmin{A} = \frac{\curl B_A}{\pmin^2}
		\end{equation} 
		for some $B_A$ which satisfies $\diver B_A = 0$ and $B_A \times \nu = 0$. We denote $B^0_\ep=B_{A_\ep^0}$ and $\phi_\ep^0=\phi_{A_\ep^0,\pmin^2}$. Notice that $\phi_\ep^0$ is phase that appears in \eqref{meissnerconf}.
        
        As we shall see next, the vector field $B^0_\ep$ is the 3D analog of the function $\xi_\ep$ appearing in \cite{diaz-roman-2d} in the 2D framework. It belongs to the space $C_T^{0,\gamma}(\Omega,\R^3)$ of $\gamma$-Hölder-continuous vector fields whose tangential component vanishes on $\partial \Omega$, for any $\gamma\in(0,1]$. The regularity of the vector field $B^0_\ep$ mainly depends on the regularity of $H_{0,\ex}$ in $\Omega$. In particular, assuming henceforth $H_{0,\ex} \in L^3(\Omega,\R^3)$, we have that 
		$$
		\|B_\ep^0\|_{C_T^{0,\gamma}(\Omega,\R^3)}\leq C\quad \mbox{for any }\gamma\in(0,1),
        $$
        where the constant $C>0$ does not depend on $\ep$, which is similar to the regularity results for $\xi_\ep$ in 2D. This estimate is proved in Proposition~\ref{prop:regularityB0} and is fundamental to our analysis.
        
        It is worth noting that in the 2D setting, we were able to establish that the previous bound holds for $\gamma=1$. We believe this to be the case in 3D as well, although our method of proof does not extend to the case $\gamma=1$ in 3D. Nevertheless, this limitation does not affect our analysis.
	\end{itemize}
	A crucial result concerning the special configuration \eqref{meissnerconf} is the energy splitting that we present next, which is analogous to \cite{diaz-roman-2d}*{Proposition 1.1}. Before doing so, we must define the appropriate space for the minimization of $GL_\ep$. We will assume throughout the article that $H_{\ex}\in L_{\loc}^2(\R^3,\R^3)$. Since $\diver H_{\ex}=0$ in $\R^3$, in accordance with the nonexistence of magnetic monopoles in Maxwell’s theory of electromagnetism, we deduce that there exists a vector potential $A_\ex\in H^1_{\loc}(\R^3,\R^3)$ such that
    $$
    \curl A_{\ex}=H_{\ex}\quad \mathrm{and}\quad \diver A_{\ex}=0\ \mathrm{in}\ \R^3.
    $$
    We also define $A_{0,\ex}\colonequals h_\ex^{-1}A_\ex$. Observe that we can choose $A_\ex$ in the Coulomb gauge, that is, 
    $$\left\{\begin{array}{rcll}
        \diver A_\ex &=& 0 &\ \mathrm{in}\ \Omega \\
        A_\ex \cdot \nu &=& 0 &\ \mathrm{on}\ \partial \Omega. 
    \end{array}\right.$$
    This special choice for $A_\ex$ yields the following useful estimate in $\Omega$, for any $p \in (1,\infty)$,
    $$\norm{A_\ex}{W^{1,p}(\Omega,\R^3)} \leq C \norm{H_\ex}{L^p(\Omega,\R^3)},$$
    where $C=C(\Omega,p)>0$. Similarly,
    \begin{equation}\label{aex cgauge estimate}
        \norm{A_{0,\ex}}{W^{1,p}(\Omega,\R^3)} \leq C \norm{H_{0,\ex}}{L^p(\Omega,\R^3)}.
    \end{equation}
    
    The natural space for the minimization of $GL_\ep$ in three dimensions is 
    $$
    H^1(\Omega,\C)\times [A_{\ex}+H_{\curl}],
    $$
    where
    $$
    H_{\curl}\colonequals \{A\in H^1_{\loc}(\R^3,\R^3) \ | \ \curl A\in L^2(\R^3,\R^3)\}.
    $$
    We are now ready to present our energy splitting.
	\begin{prop}\label{prop:energy splitting}
		Given any configuration $(\u,\A)\in \minsp$, letting $(u,A)$ be defined via the relation $(\u,\A) = (\pmin u e^{i h_\ex \phi_\ep}, A+h_\ex A^0_\ep)$, we have
		\begin{equation}\label{energy splitting equation}
			GL_\ep(\u,\A) = GL_\ep(\pmin e^{i h_\ex \phi_\ep}, h_\ex A^0_\ep) + \fen(u,A) + \frac{1}{2} \int_{\R^3 \setminus \Omega} |\curl A|^2 - h_\ex \int_\Omega \mu(u,A) \cdot B^0_\ep + \err,
		\end{equation}
		where
		\begin{equation*}
			\err \colonequals \frac{h_\ex^2}{2}\int_\Omega \frac{|\curl B^0_\ep|^2}{\pmin^2}(|u|^2-1).
		\end{equation*}
	\end{prop}
	This splitting was largely driven by a combination of ideas that played a key role in developing the energy splittings presented in \cites{Rom-CMP,diaz-roman-2d}.
	
	\medskip	
	Let us remark that $\err=\err(\ep)$ is negligible, in particular, when $h_\ex=O(|\log\ep|^m)$ for any $m>0$. The first term in the RHS of \eqref{energy splitting equation} captures, with high precision, the minimal energetic cost among vortexless pairs. We have
	$$
	GL_\ep(\meisconf) = E_\ep(\pmin) + h_\ex^2 J_\ep(A^0_\ep),
	$$
	where $E_\ep(\pmin)$ can be interpreted as the energetic cost ``enforced'' by the potential term in \eqref{GLenergy}, whereas the latter term $h_\ex^2 J_\ep(A^0_\ep)$ captures the energetic cost associated to the external field.
	
	Let us also observe that, since $\rho_\ep\geq \sqrt{b}>0$, $\mathbf u$ and $u$ have the same vortices and $\mu(\mathbf u,\mathbf A)\approx \mu(u,A)$. For this reason, the second term on the RHS of \eqref{energy splitting equation} can be interpreted as the energetic cost of the vortex filaments, while the third term represents the magnetic gain associated with them. Hence, the occurrence of vortices strongly depends on the sign of
	\begin{equation}\label{sign}
		F_{\varepsilon,\pmin}(u,A) - h_{\ex} \int_\Omega \mu(u,A) \cdot B^0_\ep.
	\end{equation}
	
	\subsection{\texorpdfstring{$\ep$}{ε}-level estimates} To analyze the sign of \eqref{sign}, the $\ep$-level estimates developed in \cite{roman-vortex-construction}*{Theorem 1.1} become essential. However, their direct application in our context is hindered by the presence of the weight in the free energy $\fen(u,A)$. Our first result overcomes this limitation by generalizing the $\ep$-level tools to the weighted framework.	
	
	\begin{theorem}\label{epleveltools} Assume $a_\ep$ is such that:
		\begin{enumerate}[leftmargin=*,itemsep=3pt,label=\normalfont(\roman*)]
			\item There exist $\alpha\in (0,1)$, $N>0$, and $C_1>0$ (that do not depend on $\ep$) such that
			$$
			\|\pmin\|_{C^{0,\alpha}(\Omega)}\leq C_1|\log\ep|^N.
			$$
			\item There exist $C_2>0$ and $\kappa>0$ (that do not depend on $\ep$) such that $a_\ep$ is constant in the set $\{ x\in \Omega \ | \ \dist(x,\partial\Omega)\leq C_2|\log\ep|^{-\kappa})\}$.
		\end{enumerate}
		Then, the following holds. For any $m,n,M>0$ there exist $C,\ep_0>0$ depending only on $M,m,n$, and $\partial \Omega$ such that, for any $\ep < \ep_0$, if $(u,A)\in H^1(\Omega,\C)\times H^1(\Omega,\R^3)$ satisfies $\fen(u,A) \leq M|\log \ep|^m$, then there exists a polyhedral 1-current $\nu_\ep$ such that:
		\begin{enumerate}[leftmargin=*,itemsep=3pt,label=\normalfont(\arabic*)]
			\item $\nu_\ep/\pi$ is integer multiplicity,
			\item $\partial \nu_\ep = 0$ relative to $\Omega$,
			\item $\mathrm{supp}(\nu_\ep)\subset S_{\nu_\ep}\subset \overline\Omega$ with $|S_{\nu_\ep}| \leq C|\log \ep|^{-q}$, where $q(m,n) \colonequals 3(m+n)$,
			\item
			\begin{equation}\label{weighted lower bound}
				\int_{S_{\nu_\ep}} \pmin^2 |\nabla_A u|^2 + \frac{\pmin^4}{2\ep^2}(1-|u|^2)^2 + |\curl A|^2 \geq |\pmin^2 \nu_\ep|\left(|\log \ep| - C\log|\log \ep| \right) - \frac{C}{|\log \ep|^n},
			\end{equation}
			where $|\pmin^2\nu_\ep|$ denotes the mass of the $1$-current $\pmin^2\nu_\ep$ in $\Omega$\footnote{See Section~\ref{sec:iso} for the notation on currents.},
			\item and for any $\gamma \in (0,1]$ there exists $C_\gamma >0 $ depending only on $\gamma$ and $\partial \Omega$ such that
			\begin{equation}\label{vorticity estimate}
				\norm{\mu(u,A) - \nu_\ep}{(C^{0,\gamma}_T(\Omega,\R^3))^*} \leq C_\gamma \frac{\fen(u,A)+1}{|\log \ep|^{q\gamma}},
			\end{equation}
			where $C_T^{0,\gamma}(\Omega,\R^3)$ denotes the space of $\gamma$-H\"older continuous vector fields defined in $\Omega$ whose tangential component vanishes on $\partial \Omega$. 
			
		\end{enumerate}
	\end{theorem}
    \begin{remark}
	The hypotheses on $a_\ep$ are primarily technical and were introduced to facilitate the adaptation of the construction from \cite{roman-vortex-construction} to the weighted framework. If they do not hold, a different (and more intricate) lower bound is obtained in place of \eqref{weighted lower bound}; see Remark~\ref{remark: hyp a_ep not satisfied}. The rest of the statement remains valid in this case.
    \end{remark}
    
	Using this theorem, we can estimate the difference \eqref{sign} from below, which in turn provides a lower bound for the value of $h_\ex$ at which the sign of the expression becomes negative, triggering the onset of vortex filament formation. This leads to a weighted version of the \emph{isoflux} problem introduced and analyzed in \cites{Rom-CMP,roman-sandier-serfaty}.
	
	\subsection{The weighted isoflux problem and main result}\label{sec:iso}
	Given a domain $\Omega \subset \R^3$, 
	we let $\NN$ be the space of normal $1$-currents supported in $\overline\Omega$, with boundary supported on $\partial\Omega$. We denote 
	by $|\cdot |$ the mass of a current. Recall that normal currents are currents with finite mass, whose boundaries have finite mass as well.
	We also let $X$ denote the class of currents in $\NN$ that are simple oriented Lipschitz curves. 
	An element of $X$  must either be a loop contained in $\overline\Omega$ or have its two endpoints on $\partial \Omega$. 
	
	For any vector field $B\in C_T^{0,1}(\Omega,\R^3)$ and any $\ga\in\NN$ we denote by $\ip{\ga}{B}$ the value of $\ga$ applied to $B$, which corresponds to the circulation of the vector field $B$ on $\ga$ when $\ga$ is a curve. If $\mu$ is a $2$-form then $\ip{\mu}{B}$ is the scalar product of the $2$-form $\mu$ and $B$, seen as  a $2$-form in this case, in $L^2(\Omega,\Lambda^2(\R^3))$. Let us now recall the definition of the isoflux problem.
	
	\begin{define}[Isoflux problem]
		The isoflux problem relative to $\Omega $ and a vector field $B\in C_T^{0,1}(\Omega, \R^3)$,  is  the question of maximizing  over $\NN$ the ratio 
		\begin{equation*}
			\RR(\ga):=\dfrac{\ip{\ga}{B}}{|\ga|}.
		\end{equation*}
	\end{define}
	\begin{remark}
		In the homogeneous case $a_\ep\equiv 1$ in $\Omega$, which implies $\pmin\equiv 1$ in $\Omega$, the constant $\RR_0$ that appears in \eqref{firstweak} corresponds to the supremum over $\NN$ of the ratio $\RR$ for the special vector field $B_0$ which appears in the Hodge decomposition of the minimizer $A_0$ of the functional defined in \eqref{reducedenergy} when $\pmin\equiv 1$ in $\Omega$, that is, $B_0$ satisfies \eqref{b0 existence} with $\pmin\equiv 1$ in $\Omega$.
	\end{remark}
	
	Now, given a smooth function $\eta$ and $\ga\in \NN$, we define $\eta \Gamma\in \NN$ by duality, that is, for any 1-form $\omega$ supported in $\Omega$, we let
	\begin{equation*}
		\ip{\eta \Gamma}{\omega} \colonequals \ip{\Gamma}{\eta \omega}.
	\end{equation*} 
	
	\begin{define}[Weighted isoflux problem]\label{def weighted isoflux}
		The weighted isoflux problem relative to $\Omega $, a smooth weight function $\eta$, and a vector field $B\in C_T^{0,1}(\Omega, \R^3)$, is  the question of maximizing  over $\NN$ the ratio 
		\begin{equation*}
			\RR_\eta(\ga):=\dfrac{\ip{\ga}{B}}{|\eta\ga|}.
		\end{equation*}
	\end{define}
	\begin{remark}
	The existence of maximizers for $\RR_\eta$ is ensured by the weak-$\star$ sequential compactness of $\NN$. The argument is the same as the one used in the proof of \cite{roman-sandier-serfaty}*{Theorem 1} when $\eta \equiv 1$.
	\end{remark}
    
	Hereafter we will only be interested in the case $B=B_\ep^0$, where we recall that $B_\ep^0$ is the vector field defined via \eqref{b0 existence} for $A=A_\ep^0$.
	
	By combining the $\ep$-level tools from Theorem \ref{epleveltools} with the energy splitting \eqref{energy splitting equation}, we deduce that the occurrence of vortex filaments is possible only if $h_\ex\geq H_{c_1}^\ep$, where
	\begin{equation}\label{def RRep}
	H_{c_1}^\ep\colonequals \frac{|\log\ep|}{2\RR_\ep}\quad \mbox{and}\quad 
	\RR_\ep\colonequals\sup_{\ga\in X}\RR_{\pmin^2}(\ga)=\sup_{\ga\in X}\frac{\ip{\ga}{B_\ep^0}}{|\pmin^2\ga|}.
	\end{equation}
	The following is our main result.
	\begin{theorem}\label{main result lower bound}
	Suppose $\liminf_{\ep \to 0^+} \RR_\ep> 0$. If the hypotheses on $a_\ep$ in Theorem \ref{epleveltools} hold, then there exist $\ep_0, K_0 > 0$ such that, for any $\ep < \ep_0$ and $h_\ex \leq H_{c_1}^\ep - K_0 \log|\log \ep|$, the global minimizers $(\u,\A)$ of $GL_\ep$ in $\minsp$ are such that, letting $(u,A)=\left(\pmin^{-1}\u e^{-i h_\ex \phi_\ep},\A-h_\ex A_\ep^0\right)$, as $\ep\to 0$, we have
	\begin{enumerate}[leftmargin=*,label={\normalfont (\alph*)}]
		\item\label{item weakvort} $\norm{\mu(u,A)}{\left(C^{0,\gamma}_T(\Omega,\R^3\right)^*} = o(1)$ for any $\gamma \in(0,1]$.
		\item\label{item meisapprox} $GL_\ep(\u,\A) = GL_\ep\left(\pmin e^{i h_\ex \phi_\ep^0},h_\ex A^0_\ep\right) +o(1)$.
	\end{enumerate}
	Moreover, if $(\u,\A) \in \minsp$ is a configuration such that $|\u|>c>0$ for some $c>0$, $GL_\ep(\u,\A)\leq GL_\ep\left(\meisconf\right)$, and $h_\ex \leq \ep^{-\alpha}$ for some $\alpha \in \left(0,\frac{1}{2}\right)$, then, as $\ep \to 0$,
	$$
	GL_\ep(\u,\A)=GL_\ep\left(\meisconf\right)+o(1).
	$$
	\end{theorem}
	The first part of this result characterizes the behavior of global minimizers of $GL_\ep$ below $H_{c_1}^\ep$, showing in particular that
	$$
	H_{c_1}^\ep-K_0\log|\log\ep|\leq H_{c_1}.
	$$ 
	It essentially states that global minimizers are weakly vortexless, in the sense that their vorticities tend to zero as $\ep\to 0$. Moreover, the minimal energy is, up to a small error, equal to the energy of our approximation of the Meissner state.
	
	In contrast, in the weighted 2D and homogeneous 3D settings, we were able to prove a stronger result. Specifically, we showed that (see \cite{diaz-roman-2d}*{Theorem 1.1} and \cite{Rom-CMP}*{Theorem 1.2})
	$$
	\|1-|u|\|_{L^\infty(\Omega,\C)}=o(1).
	$$
	This is due to the absence of clearing out in our weighted 3D setting, which constitutes an interesting and challenging open question that may be addressed independently of the present analysis. 
	
	\medskip
	The final part of our main result establishes that our approximation of the Meissner state is energetically optimal among vortexless configurations, even when $h_\ex$ significantly exceeds $O(|\log \ep|)$. In particular, from our energy splitting, we deduce that to obtain a matching upper bound for the first critical field $H_{c_1}$, it suffices to construct a configuration $(u,A)$ such that \eqref{sign} is negative. One may even take $A=0$. To achieve this, the main idea is to construct $u$ with a single vortex filament located on an ($\ep$-dependent) $1$-current that is nearly optimal for the weighted isoflux problem. This, too, presents a challenging open problem that merits independent investigation.
	
	\medskip
	Finally, we have the following result concerning the first hypothesis of our main result. 
    \begin{prop}\label{prop:conditionliminf}
        Assume $\curl H_{0,\ex} \not\equiv 0$ in $\Omega$. Then, we have
        \begin{equation*}
            \liminf_{\ep \to 0^+} \RR_\ep> 0.
        \end{equation*}
    \end{prop}
	
	\subsection*{Outline of the paper} The rest of the paper is organized as follows. In Section \ref{section meissner} we provide the tools to construct the approximating Meissner configuration and prove the energy splitting result, Proposition \ref{prop:energy splitting}. Section~\ref{section isoflux} presents results on lower bounds for solutions of the weighted isoflux problem and establishes Proposition~\ref{prop:conditionliminf}. In Section \ref{section main result} we prove our main result, Theorem \ref{main result lower bound}. In Section \ref{section eplevel} we introduce the appropriate modifications for the proof of the $\ep$-level estimates in the inhomogeneous setting, Theorem \ref{epleveltools}.
	
	\subsection*{Acknowledgments} This work was partially funded by ANID FONDECYT 1231593.

\section{Approximating Meissner configuration and Energy splitting}\label{section meissner}
The main goal of this section is to construct the previously mentioned approximation of the Meissner state given by the configuration $\left(\meisconf\right)$ and prove the associated energy splitting result stated in Proposition~\ref{prop:energy splitting}. We start by giving an overview on a crucial decoupling result regarding $\pmin$, the unique positive minimizer of the energy functional $E_\ep$ defined in \eqref{energywithoutmagneticterm}, and some classical results of the $L^p$-theory of vector fields, which are employed to construct and characterize this special configuration. 

\subsection{Decoupling of the pinning function.}
The following decoupling result establishes a crucial link between inhomogeneous and homogeneous configurations, enabling the generalization of many tools originally developed for the homogeneous case. We observe that $E_\ep$ is both coercive and weakly lower semicontinuous in $H^1(\Omega,\C)$. This ensures the existence of a minimizer $u_\ep$ for $E_\ep$. Moreover, since $E_\ep(|u_\ep|) \leq E_\ep(u_\ep)$, we deduce that $E_\ep$ admits a positive minimizer. The elliptic structure of minimizers of $E_\ep$ given by its Euler--Lagrange equation \eqref{pde rho} allows us to deduce the following decoupling lemma.
\begin{lemmaA}
    Let $u \in H^1(\Omega,\C)$ and $\pmin$ be a positive minimizer of $E_\ep$.
    \begin{enumerate}[leftmargin=*,label={\normalfont (\arabic*)}]
        \item It holds that
        \begin{equation}\label{lm nomag decoupling}
            E_\ep(\pmin u) = E_\ep(\pmin) + \frac{1}{2}\int_\Omega \pmin^2|\nabla u|^2 + \frac{\pmin^4}{2\ep^2}(1-|u|^2)^2.
        \end{equation}

        \item Defining the energy with magnetic field by
        \begin{equation*}
            E_\ep(u,A) \colonequals \frac{1}{2}\int_\Omega |\nabla_A u|^2 + \frac{1}{2\ep^2}(a_\ep-|u|^2)^2,
        \end{equation*}
        we have
        \begin{equation}\label{lm mag decoupling}
            E_\ep(\pmin u, A) = E_\ep(\pmin) + \frac{1}{2}\int_\Omega \pmin^2 |\nabla_A u|^2 + \frac{\pmin^4}{2\ep^2}(1-|u|^2)^2.
        \end{equation}
    \end{enumerate}
\end{lemmaA}
The result for the energy without magnetic field given in \eqref{lm nomag decoupling}, was first introduced by Lassoued and Mironescu \cite{lassoued-mironescu} in two dimensions, under a specific pinning function and a prescribed Dirichlet boundary condition. Basically, the same proof applies to any dimension and for any pinning function $a_\ep$, even when modifying the boundary condition for $\pmin$. We refer to \cite{diaz-roman-2d}*{Lemma 2.1} for a proof of \eqref{lm mag decoupling} in the 2D setting, which directly extrapolates to the 3D one.

\medskip
A consequence of \eqref{lm nomag decoupling} is the following result.
\begin{lemmaA}
There exists a unique positive minimizer $\pmin$ for $E_\ep$. Moreover, any other minimizer for $E_\ep$ is of the form $\pmin e^{i\theta}$, where $\theta$ is a real constant.
\end{lemmaA}
The uniqueness of the positive minimizer was also proven in \cite{lassoued-mironescu}, but we refer to \cite{DosRodSan}*{Corollary 2.1} for a proof in our context.

\medskip
Thus, from now on, $\pmin$ denotes the unique positive minimizer of $E_\ep(\cdot)$. This function can be interpreted as a regularized version of $\sqrt{a_\ep}$, which it approximates well. In particular, we have the following result. 
\begin{propA}\label{prop:rho_ep like a_ep}
    There exist $C,c>0$ independent of $\ep$ such that for any $x\in \Omega, R>0$, if $a_\ep$ is constant in $\Omega \cap B(x,R)$, then 
    \begin{equation*}
        \sup_{\Omega \cap B(x,R)} |\sqrt{a_\ep} - \pmin| \leq C e^{-c\frac{R}{\ep}}.
    \end{equation*}
\end{propA}
We refer to \cite{DosSantos-Misiats}*{Appendix A} for a proof.

\subsection{Notation and some classical results on vector-valued Sobolev spaces}
In this subsection, we state fundamental results on vector fields that will be used to prove the existence and regularity of the approximation of the Meissner configuration. We refer to \cites{kozono-yanagisawa-decomposition,ky-exterior} and the references therein for a more detailed analysis on the $L^p$-theory for vector fields on bounded and exterior domains, respectively. First, let us introduce some important spaces that will appear throughout this section:
\begin{itemize}
    \item The homogeneous Sobolev space $\dot{H}^{1}(\R^3,\R^3)$ is the completion of $C^{\infty}_0(\R^3,\R^3)$ with respect to the norm $\norm{\nabla (\cdot)}{L^2(\R^3, \R^3)}$. A classical result from Ladyzhenskaya \cite{ladyzhenskaya} shows that in this space the functional 
    $$\left(\norm{\diver (\cdot)}{L^2(\R^3,\R^3)}^2 + \norm{\curl (\cdot)}{L^2(\R^3,\R^3)}^2\right)^{\frac{1}{2}}$$ 
    defines a norm equivalent to the standard $H^{1}$-norm. The space $\hsolhom^{1}(\R^3,\R^3)$ denotes the space of solenoidal (divergence-free) vector fields in $\dot{H}^{1}(\R^3,\R^3)$. The inherited norm for this space is just $\norm{\curl (\cdot)}{L^2(\R^3,\R^3)}$. 

    \item The space $\lsol^2(\Omega,\R^3)$ is the space of vector fields in $L^2(\Omega,\R^3)$ with zero divergence in the sense of distributions, that is, $A \in \lsol^2(\Omega,\R^3)$ if and only if $A \cdot \nu = 0$ on $\partial \Omega$ and for any $\phi \in C^\infty_0(\Omega)$
    $$\int_\Omega A \cdot \nabla \phi = 0.$$

    \item The space $\hsolt^{1}(\Omega,\R^3)$ is the space of vector fields $A \in H^{1}(\Omega,\R^3)$ such that $\diver A = 0$ and $A \times \nu =0$ on $\partial \Omega$. 
\end{itemize}

The following are classical characterizations of $\lsol^2(\Omega,\R^3)$ in the case when $\Omega$ does not have holes.
\begin{lemmaA}[Properties of $\lsol^2(\Omega,\R^3)$] The following holds.
    \begin{enumerate}[leftmargin=*,label=\normalfont (\arabic*)] 
        \item The space of smooth compactly supported vector fields with zero divergence is dense in $\lsol^2(\Omega,\R^3)$, that is,
        \begin{equation}\label{prop lsol characterizations dense}
            \overline{\{\Phi \in C^\infty_{0}(\Omega,\R^3) \colon \diver \Phi = 0 \} }^{\norm{\cdot}{L^2(\Omega,\R^3)}} = \lsol^2(\Omega,\R^3).
        \end{equation}
        \item There exists a unique $B_A \in \hsolt^{1}(\Omega,\R^3)$ such that \begin{equation}\label{prop lsol characterizations curl}
            A = \curl B_A.
        \end{equation} Moreover, we have 
        \begin{equation*}
            \norm{B_A}{H^{1}(\Omega,\R^3)} \leq C \norm{A}{L^2(\Omega,\R^3)}
        \end{equation*}
        for some $C = C(\Omega) > 0$.
    \end{enumerate}
\end{lemmaA}

A powerful tool in the analysis of solenoidal and potential fields is the \emph{Helmholtz--Hodge decomposition}, which states that for any bounded or exterior smooth domain $\Omega \subseteq \R^3$
\begin{equation}\label{hodge direct sum of spaces}
    L^2(\Omega,\R^3) = \lsol^2(\Omega,\R^3) \oplus \nabla H^{1}(\Omega).
\end{equation}
We will use the following forms of \eqref{hodge direct sum of spaces} given by the geometry of the domain.
\begin{propA}[Helmholtz--Hodge decomposition] The following holds.

    \begin{enumerate}[leftmargin=*,label=\normalfont (\arabic*)] 
    \item For every $A \in L^2(\Omega, \R^3)$, there exists a unique pair $(B_A, \phi_A)$, $B_A \in \hsolt^{1}(\Omega, \R^3)$ and $\phi_A \in H^{1}(\Omega)$ with $\int_\Omega \phi_A = 0$ such that
            \begin{equation}\label{hodge decomposition}
                     A = \curl B_A + \nabla \phi_A 
            \end{equation}
            This decomposition is continuous in $L^2(\Omega,\R^3)$, that is, there exists $C = C(\Omega) > 0$ such that for any $A \in H^1(\Omega,\R^3)$
            \begin{equation*}
                \norm{\nabla \phi_A}{L^2(\Omega)} + \norm{\curl B_A}{L^2(\Omega, \R^3)} \leq C \norm{A}{L^2(\Omega, \R^3)}.
            \end{equation*}

        \item For every $A \in L^2(\R^3, \R^3)$, there exists a unique pair $(B_A, \phi_A)$, $B_A \in \hsolhom^1(\R^3, \R^3)$ and $\phi_A \in H^1_\mathrm{loc}(\R^3)/\R$ with $\nabla \phi_A \in L^2(\R^3,\R^3)$ such that 
            \begin{equation}\label{hodge decomposition r3}
                     A = \curl B_A + \nabla \phi_A.
            \end{equation}
            This decomposition is continuous in $L^2(\R^3,\R^3)$, that is, there exists $C > 0$ such that for any $A \in H^{1}(\R^3,\R^3)$
            \begin{equation*}
                \norm{\nabla \phi_A}{L^2(\R^3,\R^3)} + \norm{\curl B_A}{L^2(\R^3, \R^3)} \leq C \norm{A}{L^2(\R^3, \R^3)}.
            \end{equation*}
    \end{enumerate}
    \begin{remark}\label{hodge lp}
        The decomposition result for bounded domains \eqref{hodge decomposition} also holds if we replace $L^2$ by $L^p$, for any $p \in (1,\infty)$.
    \end{remark}
\end{propA} 

\subsection{Construction of the approximating Meissner configuration}
The Lassoued--Mironescu decoupling \eqref{lm mag decoupling} gives us an intuition on how to identify vortexless configurations. Since $\frac{u}{\pmin}$ behaves like a homogeneous configuration, vortexless configurations should satisfy $\left|\frac{u}{\pmin}\right|^2 \approx 1$, that is, $|u| \approx \pmin$. Thus, a good starting point is to find $A$ in an adequate space such that it minimizes $GL_\ep(\pmin,A)$. However, for reasons that will be explained later (see Remark \ref{remark choice of phi}), we will introduce a function $\phi_A$ and minimize $GL_\ep(\pmin e^{i\phi_A},A) = GL_\ep(\pmin, A-\nabla \phi_A)$ instead. We can use the gauge-invariance property of $GL_\ep$ and the Helmholtz--Hodge decomposition \eqref{hodge decomposition r3} to assume that $A-A_{0,\ex} \in \hsolhom^1(\R^3,\R^3)$ a priori. This particular choice of gauge additionally guarantees coercivity with respect to $A$.

\medskip
For convenience's sake, we rescale to minimize $GL_\ep(\pmin, h_\ex(A-\nabla \phi_A))$. Observe that, from \eqref{lm mag decoupling}, we have
\begin{align*}
    GL_\ep(\pmin,h_\ex (A-\nabla \phi_A)) &= GL_\ep(\pmin \cdot 1,h_\ex (A-\nabla \phi_A)) \\
    &= E_\ep(\pmin) + \frac{h_\ex^2}{2}\int_\Omega \pmin^2|A-\nabla \phi_A|^2 +  \frac{h_\ex^2}{2}\int_{\R^3}|\curl A - H_{0,\ex}|^2.
\end{align*}
Therefore, minimizing $GL_\ep(\pmin,h_\ex(A-\nabla \phi_A))$ is equivalent to minimizing 
$$\frac{1}{2} \int_\Omega \pmin^2 |A-\nabla \phi_A|^2 + \frac{1}{2} \int_{\R^3} |\curl (A - A_{0,\ex})|^2.$$
Finally, we make the crucial choice (see Remark \ref{remark choice of phi}) $\phi_A=\phipmin{A}$, the unique solution in $H^1(\Omega)$ with zero-average of \eqref{phi rho equation}, which leads us to consider the functional $J_\ep$, defined for $A \in [A_{0,\ex} + \Dot{H}^1_\sigma(\R^3,\R^3)]$ by
$$J_\ep(A) \colonequals \frac{1}{2} \int_\Omega \pmin^2 |A-\nabla \phipmin{A}|^2 + \frac{1}{2} \int_{\R^3} |\curl (A - A_{0,\ex})|^2.$$

\medskip
A standard analysis of this functional yields the following result.
\begin{prop}\label{existence of minimizer}
     $J_\ep$ admists a unique minimizer $A^0_\ep \in  [A_{0,\ex} + \hsolhom^{1}(\R^3,\R^3)]$ which satisfies.
    \begin{enumerate}[leftmargin=*,label={\normalfont (\alph*)}]
        \item\label{item1} $J_\ep(A^0_\ep) \leq C \norm{H_{0,\ex}}{L^2(\Omega,\R^3)}^2$ for some $C=C(\Omega)>0$

        \item\label{item2} The Euler--Lagrange equation solved by the minimizer $A^0_\ep$ is given by
        \begin{equation}\label{b0 r3 euler lagrange}
            \curl (\curl A^0_\ep - H_{0,\ex}) + \chi_\Omega \curl \bpmin{A^0_\ep} = 0\quad \mathrm{in }\ \R^3,
        \end{equation}
        where $\bpmin{A^0_\ep} \in \hsolt^1(\Omega,\R^3)$ is the vector field related to $A^0_\ep$ by \eqref{b0 existence} and $\chi_\Omega$ denotes the indicator function of the set $\Omega$.
        
        \item\label{item3} The vector field $\bpmin{A^0_\ep}$ satisfies the following variational equation
        \begin{equation}\label{b0 integral equation}
            \left\{\begin{array}{rcll}
              \displaystyle \int_\Omega \left(\curl \dfrac{\curl \bpmin{A^0_\ep}}{\pmin^2} + \bpmin{A^0_\ep} \right) \cdot V  &=& \displaystyle \int_\Omega H_{0,\ex} \cdot V &\mathrm{in}\ \Omega\\
               \bpmin{A^0_\ep} \times \nu &=& 0 &\mathrm{on}\ \Omega,
            \end{array}\right.
        \end{equation}
        for all $V \in \lsol^2(\Omega,\R^3)$. 
    \end{enumerate}
\end{prop}

Before proving this proposition, let us introduce some notation regarding the minimizer $A^0_\ep$ that will be used throughout the rest of the paper. We denote $H^0_\ep\colonequals \curl A^0_\ep$. Also, from \eqref{phi rho equation} and \eqref{b0 existence} for $A=A^0_\ep$, we get a function $\phipmin{A^0_\ep}$ and a vector field $\bpmin{A^0_\ep}$, which we will denote respectively by $\phi_\ep^0$ and $B^0_\ep$.

\begin{proof}
    We split the proof into two steps.
    \begin{enumerate}[label=\textsc{\bf Step \arabic*.},leftmargin=0pt,labelsep=*,itemindent=*,itemsep=10pt,topsep=10pt]
    \item {\bf Existence and uniqueness. Proof of item \ref{item1}.} Since the standard $H^1$-norm in $\hsolhom^1(\R^3,\R^3)$ is equivalent to $\norm{\curl A}{L^2(\R^3,\R^3)}$, the coercivity of $J_\ep$ in $[A_{0,ex} + \hsolhom^{1}(\R^3,\R^3)]$ follows directly. On the other hand, by the linearity of equation \eqref{phi rho equation}, we deduce that the map $A \to \nabla \phipmin{A}$ is linear. Furthermore, by using $\phipmin{A}$ as a test function in $\eqref{phi rho equation}$, we have
    $$ \int_\Omega \pmin^2 |\nabla \phipmin{A}|^2 = \int_\Omega \pmin^2 A \cdot \nabla \phipmin{A} . $$
    By the Cauchy-Schwarz inequality and \eqref{rho uniform bound}, we conclude that the map $A \to \nabla \phipmin{A}$ is a bounded linear operator from $L^2(\Omega,\R^3)$ to $L^2(\Omega,\R^3)$. Therefore, $J_\ep$ is both (strongly) continuous and strictly convex, which implies that it is also weakly lower semicontinuous. It thus follows from the direct method of the calculus of variations that $J_\ep$ admits a unique minimizer in $[A_{0,\ex} + \hsolhom^{1}(\R^3,\R^3)]$ denoted by $A^0_\ep$. Finally, we see that item \ref{item1} is satisfied by noticing that $A_{0,\ex}$ is admissible for $J_\ep$, since, by recalling the boundedness of the operator $A \to \nabla \phipmin{A}$, we have
    \begin{align*}
        J_\ep(A^0_\ep) \leq J_\ep(A_{0,\ex}) &= \frac{1}{2}\int_\Omega \pmin^2|A_{0,\ex} - \nabla \phipmin{{A_{0,\ex}}}|^2 \\
        &\stackrel{\eqref{rho uniform bound}}{\leq} C \norm{A_{0,\ex}}{L^2(\Omega,\R^3)}^2\\
        &\stackrel{\eqref{aex cgauge estimate}}{\leq} C \norm{H_{0,\ex}}{L^2(\Omega,\R^3)}^2.
    \end{align*}
    
    \item {\bf Euler--Lagrange equation for $A^0_\ep$. Proof of items \ref{item2} and \ref{item3}.} From minimality of $A^0_\ep$, we  deduce that it satisfies, for all $A \in \hsolhom^{1}(\R^3,\R^3)$,
    \begin{equation}\label{weak euler lagrange a0}
        \int_\Omega \pmin^2(A^0_\ep - \nabla \phi_\ep^0) \cdot (A - \nabla \phipmin{A}) + \int_{\R^3} (H^0_\ep - H_{0,\ex}) \cdot \curl A = 0.
    \end{equation}
    Since $\pmin^2(A^0_\ep - \nabla \phi_\ep^0) \in \lsol^2(\Omega,\R^3)$, it follows from \eqref{prop lsol characterizations curl} that we can write $\pmin^2(A^0_\ep - \nabla \phi_\ep^0)$ as $\curl B^0_\ep$, where $B^0_\ep \in \hsolt^{1}(\Omega,\R^3)$. It is worth noting that $B^0_\ep$ satisfies
    \begin{equation}\label{b0 properties}
        \left\{\begin{array}{rcll}
         \diver B^0_\ep &=& 0 &\mathrm{in}\ \Omega  \\
          \curl B^0_\ep &=& \pmin^2(A^0_\ep - \nabla \phi_\ep^0) &\mathrm{in}\ \Omega \\
          B^0_\ep \times \nu &=& 0 &\mathrm{on}\ \partial \Omega\\
          \curl B^0_\ep \cdot \nu &=& 0 &\mathrm{on}\ \partial \Omega.
    \end{array}
    \right.
    \end{equation}
    Additionally, integrating by parts we deduce the orthogonality relation, for any $\phi \in H^{1}(\Omega)$,
    \begin{equation}\label{b0 orthogonality}
        \int_\Omega \pmin^2(A^0_\ep - \nabla \phi_\ep^0) \cdot \nabla \phi \stackrel{\eqref{b0 properties}}{=} \int_\Omega \curl B^0_\ep \cdot \nabla \phi = \int_\Omega B^0_\ep \cdot \curl \nabla \phi - \int_{\partial \Omega} \nabla \phi \cdot (B^0_\ep \times \nu) \stackrel{\eqref{b0 properties}}{=} 0.
    \end{equation}
    Using \eqref{b0 orthogonality} with $\phi = \phipmin{A}$ and inserting it into equation \eqref{weak euler lagrange a0}, we obtain
    \begin{equation}\label{weak euler lagrange b0}
        \int_\Omega \curl B^0_\ep \cdot A + \int_{\R^3} (H^0_\ep - H_{0,\ex}) \cdot \curl A = 0.
    \end{equation}
    By the Helmholtz--Hodge decomposition \eqref{hodge decomposition r3}, we conclude that $\eqref{weak euler lagrange b0}$ holds for any $A \in \Dot{H}^1(\R^3,\R^3)$. In particular, for any $\Phi \in C^\infty_0(\R^3,\R^3)$,
    $$\int_{\R^3} (\chi_\Omega \curl B^0_\ep + \curl (H^0_\ep - H_{0,\ex})) \cdot \Phi = 0. $$
    This yields that the Euler--Lagrange equation for $J_\ep$ is \eqref{b0 r3 euler lagrange}, concluding the proof of item \ref{item2}.

    Item \ref{item3} follows from using vector fields supported in $\Omega$ as test vector fields in \eqref{weak euler lagrange b0}. If $\Phi \in C^\infty_0(\Omega,\R^3)$, it follows from equation \eqref{weak euler lagrange b0} and integration by parts that
    \begin{equation*}
        \int_\Omega (B^0_\ep + H^0_\ep - H_{0,\ex}) \cdot \curl \Phi = 0.
    \end{equation*}
    Now let $\Psi \in C^\infty_0(\Omega,\R^3)$ such that $\diver \Psi = 0$ in $\Omega$. Since, by \eqref{prop lsol characterizations curl}, we can write $\Psi$ as $\curl \Phi$ in $\Omega$ with $\Phi \times \nu = 0$ on $\partial \Omega$, we have
    \begin{equation*}
        \int_\Omega (B^0_\ep + H^0_\ep - H_{0,\ex}) \cdot \Psi = 0.
    \end{equation*}
    Using the density property \eqref{prop lsol characterizations dense} of $\lsol^2(\Omega,\R^3)$ , we have that for any $A \in \lsol^2(\Omega,\R^3)$,
    \begin{equation*}
        \int_\Omega (B^0_\ep + H^0_\ep - H_{0,\ex}) \cdot A \stackrel{\eqref{b0 properties}}{=} \int_\Omega \left(B^0_\ep + \curl \dfrac{\curl B^0_\ep}{\pmin^2} - H_{0,\ex}\right) \cdot A = 0.
    \end{equation*}
    This concludes the proof of item \ref{item3}.
    \end{enumerate}
    \end{proof}
    
    \subsection{Regularity of the Meissner configuration} Although we can deduce differenciability for the minimizer using standard regularity results, we are specifically looking for uniform bounds that do not depend on $\ep$. For example, we can deduce from item \ref{item1} of Proposition \ref{existence of minimizer}, using \eqref{rho uniform bound}, that 
    \begin{equation}\label{initialboundenergy}
    \norm{\curl B^0_\ep}{L^2(\Omega,\R^3)} + \norm{A^0_\ep - A_{0,\ex}}{H^1(\R^3,\R^3)} \leq C \norm{A_{0,\ex}}{L^2(\Omega,\R^3)}\stackrel{\eqref{aex cgauge estimate}}{\leq} C \norm{H_{0,\ex}}{L^2(\Omega,\R^3)},
    \end{equation}
    where $C = C(\Omega,b) > 0$. 
    Moreover, based on the weak Euler--Lagrange equation \eqref{weak euler lagrange b0}, we obtain the following result.

    \begin{prop}\label{prop:regularityB0}
        Let $A^0_\ep$ be the minimizer of $J_\ep$ given by Proposition \ref{existence of minimizer}, and $B^0_\ep$ be given by \eqref{b0 existence}. For any $q \in [1,\infty)$, $B^0_\ep \in W^{1,q}(\Omega,\R^3)$ with
        \begin{equation}\label{b0 c01 regularity}
            \norm{B^0_\ep}{W^{1,q}(\Omega,\R^3)} \leq C \norm{H_{0,\ex}}{L^3(\Omega,\R^3)},
        \end{equation}
        where $C=C(\Omega,b,q)>0$.
        Furthermore, $B^0_\ep \in C^{0,\gamma}_T(\Omega,\R^3)$ for all $\gamma \in (0,1)$, with
        \begin{equation}\label{b0 holderreg}
        \norm{B^0_\ep}{C^{0,\gamma}(\Omega,\R^3)} \leq C\norm{H_{0,\ex}}{L^3(\Omega,\R^3)},
        \end{equation}
        where $C = C(\Omega,b,\gamma) > 0$.
    \end{prop}

    \begin{proof}
        By using interior elliptic regularity on equation \eqref{b0 r3 euler lagrange}, we have, for an open ball $B \subset \R^3$ that compactly contains $\Omega$, say $B=B(x_0,\mathrm{diam}(\Omega)+1)$ for $x_0 \in \Omega$, that there exists $C=C(\Omega)>0$ such that
        $$\norm{A^0_\ep- A_{0,\ex}}{H^{2}(B,\R^3)} \leq C \norm{\curl B^0_\ep}{L^2(\Omega,\R^3)}.$$
        In particular, we have
        \begin{align}\label{a0 elliptic regularity}
            \norm{A^0_\ep- A_{0,\ex}}{H^{2}(\Omega,\R^3)} &\leq C \norm{\curl B^0_\ep}{L^2(\Omega,\R^3)},
        \end{align}  
        where $C=C(\Omega) >0$. On the other hand, for $q \in (1,\infty)$,
        \begin{align}\label{curl b0 lp bound}
            \norm{\curl B^0_\ep}{L^q(\Omega,\R^3)} \stackrel{\eqref{b0 properties}}{=} \norm{\pmin^2(A^0_\ep - \nabla \phi_\ep^0)}{L^q(\Omega,\R^3)}
            \stackrel{\eqref{rho uniform bound}}{\leq} C\left(\norm{A^0_\ep}{L^q(\Omega,\R^3)} + \norm{\nabla \phi_\ep^0}{L^q(\Omega,\R^3)}\right).
        \end{align}
        We aim to bound $\norm{\nabla \phi_\ep^0}{L^q(\Omega,\R^3)}$ in terms of $A^0_\ep$, by exploiting the elliptic structure of equation \eqref{phi rho equation}. Using the dual characterization of the $L^q$-norm, we have, for any $q \in (1,\infty)$, that
        \begin{equation*}
            \norm{\nabla \phi_\ep^0}{L^q(\Omega,\R^3)} = \sup_{V \in L^r(\Omega,\R^3)} \frac{1}{\norm{V}{L^r(\Omega,\R^3)}} \int_\Omega V \cdot \nabla \phi_\ep^0,
        \end{equation*}
        where $r$ is such that $\frac{1}{q} + \frac{1}{r} = 1$. 
        Using the Helmholtz--Hodge decomposition \eqref{hodge decomposition} (see Remark \ref{hodge lp}), we can write $V$ as $\curl B_V+ \nabla \varphi_V$, where $\curl B_V \in \lsol^q(\Omega,\R^3)$, that is, it satisfies for any $\zeta \in W^{1,r}(\Omega)$
        $$\int_\Omega \curl B_V \cdot \nabla \zeta = 0$$
        and
        $$\norm{\curl B_V}{L^q(\Omega,\R^3)} + \norm{\nabla \varphi_V}{L^q(\Omega,\R^3)} \leq C \norm{V}{L^q(\Omega,\R^3)}$$
        for $C=C(\Omega,q)>0$. We can therefore deduce that
        $$\int_\Omega \curl B_V \cdot \nabla \phi^0_\ep =0$$
        and the continuity of the decomposition yields a constant $C=C(\Omega,q)>0$, independent of $\ep$, such that
        \begin{equation}\label{eqqqqqq}
            \norm{\nabla \phi_\ep^0}{L^q(\Omega,\R^3)} \leq C \sup_{\nabla \varphi \in L^r(\Omega,\R^3)} \frac{1}{\norm{\nabla \varphi}{L^r(\Omega,\R^3)}} \int_\Omega \nabla \varphi \cdot \nabla \phi_\ep^0. 
        \end{equation}
        Note that both the continuity and the orthogonality of the components allow us to remove the dependence on $V$ for $\nabla \varphi$. 
        
        Now, using \eqref{rho uniform bound}, from \eqref{eqqqqqq} we deduce that
        \begin{align*}
            \norm{\nabla \phi_\ep^0}{L^q(\Omega,\R^3)} &\leq C\sup_{\nabla \varphi \in L^r(\Omega,\R^3)} \frac{1}{\norm{\nabla \varphi}{L^r(\Omega,\R^3)}} \int_\Omega \nabla \varphi \cdot \pmin^2 \nabla \phi_\ep^0 \\
            \notag &\stackrel{\eqref{phi rho equation}}{=} C \sup_{\nabla \varphi \in L^r(\Omega,\R^3)} \frac{1}{\norm{\nabla \varphi}{L^r(\Omega,\R^3)}} \int_\Omega \nabla \varphi \cdot \pmin^2 A^0_\ep\\
            \notag &\leq C \norm{\pmin^2 A^0_\ep}{L^q(\Omega,\R^3)} \\
            &\stackrel{\eqref{rho uniform bound}}{\leq} C \norm{A^0_\ep}{L^q(\Omega,\R^3)}.
        \end{align*}
        Inserting this into \eqref{curl b0 lp bound}, yields
        \begin{equation*}
            \norm{\curl B^0_\ep}{L^q(\Omega,\R^3)} \leq C \norm{A^0_\ep}{L^q(\Omega,\R^3)}.
        \end{equation*}
        Thus, by the continuous embeddings $H^2(\Omega,\R^3) \hookrightarrow L^q(\Omega,\R^3)$ and $W^{1,3}(\Omega,\R^3) \hookrightarrow L^q(\Omega,\R^3)$, we have
        \begin{align*}
            \norm{\curl B^0_\ep}{L^q(\Omega,\R^3)} 
            &\leq C\left(\norm{A^0_\ep-A_{0,\ex}}{L^q(\Omega,\R^3)} + \norm{A_{0,\ex}}{L^q(\Omega,\R^3)}\right) \\
            &\leq C\left(\norm{A^0_\ep - A_{0,\ex}}{H^2(\Omega,\R^3)} + \norm{A_{0,\ex}}{W^{1,3}(\Omega,\R^3)}\right)\\
            &\stackrel{\eqref{aex cgauge estimate} \& \eqref{a0 elliptic regularity}}{\leq} C\left(\norm{\curl B^0_\ep}{L^2(\Omega,\R^3)} + \norm{H_{0,\ex}}{L^3(\Omega,\R^3)}\right)\\
            &\stackrel{\eqref{initialboundenergy}}{\leq} C\left(\norm{H_{0,\ex}}{L^2(\Omega,\R^3)} + \norm{H_{0,\ex}}{L^3(\Omega,\R^3)}\right) \\
            &\leq C\norm{H_{0,\ex}}{L^3(\Omega,\R^3)}.
        \end{align*}

        Finally, due to the boundary condition satisfied by $B^0_\ep$, we can apply a Friedrichs--Poincaré type inequality (see for instance \cite{amrouche-seloula}*{Corollary 3.2}), to obtain
        $$
            \norm{B^0_\ep}{W^{1,q}(\Omega,\R^3)}
            \leq C \left(\norm{\diver B^0_\ep}{L^q(\Omega,\R^3)} +  \norm{\curl B^0_\ep}{L^q(\Omega,\R^3)}\right) \stackrel{\eqref{b0 properties}\& \eqref{initialboundenergy}}{\leq} C\norm{H_{0,\ex}}{L^3(\Omega,\R^3)},
        $$
        where $C = C(\Omega,b,q)>0$ does not depend on $\ep$. By Sobolev embedding for any $q>3$, we have
        $$\norm{B^0_\ep}{C^{0,\gamma}(\Omega,\R^3)} \leq C \norm{B^0_\ep}{W^{1,q}(\Omega,\R^3)} \leq C \norm{H_{0,\ex}}{L^3(\Omega,\R^3)},$$
        where $\gamma = 1-\frac{3}{q}\in (0,1)$.
    \end{proof}

    \begin{remark}
        Standard elliptic regularity results for the equations satisfied by $A^0_\ep$ and $\phi_\ep^0$ yield that $B^0_\ep \in C^{0,1}_T(\Omega,\R^3)$. Nevertheless, our argument for obtaining a uniform bound in $\ep$ fails when $q=\infty$ or $\gamma = 1$. The difficulty stems from the lack of a continuous Helmholtz--Hodge decomposition \eqref{hodge decomposition} in $L^1(\Omega, \mathbb{R}^3)$, which would be required to estimate the $L^\infty$-norm of $\nabla \phi^0_\ep$ in a manner that leverages the weak formulation of equation~\eqref{phi rho equation}.
    \end{remark}

    \subsection{Proof of the Proposition~\ref{prop:energy splitting}}
    We finish this section with a proof of the energy splitting result \eqref{energy splitting equation}.

    \begin{proof}
        We start by using \eqref{lm mag decoupling} to decouple the energy in $\Omega$. We have
        \begin{multline}\label{gl after lm}
            GL_\ep(\u,\A) = E_\ep(\pmin) + \frac{1}{2}\int_\Omega \left( \pmin^2|\nabla_\A (ue^{ih_\ex \phi_\ep^0})|^2 + \frac{\pmin^4}{2\ep^2}(1-|u|^2)^2 \right) \\+ \frac{1}{2}\int_{\R^3} |\curl A + h_\ex \curl(A^0_\ep - A_{0,\ex})|^2.     
        \end{multline}

        First, by expanding the square in the second term in the RHS of \eqref{gl after lm}, we have
        \begin{align*}
            \int_\Omega \pmin^2 |\nabla_\A (ue^{ih_\ex \phi_\ep^0})|^2 &= \int_\Omega \pmin^2|\nabla_A u + ih_\ex (\nabla \phi_\ep^0 - A^0_\ep)u|^2 \\
            &\stackrel{\eqref{b0 properties}}{=} \int_\Omega \pmin^2 |\nabla_A u|^2 + h_\ex^2 \pmin^2 |\nabla \phi_\ep^0 - A^0_\ep|^2 |u|^2 - 2h_\ex \ip{iu}{\nabla_A u}\cdot \curl B^0_\ep.
        \end{align*}

        On the other hand, by expanding the square in the rightmost term of \eqref{gl after lm}, we have
        \begin{align*}
            \int_{\R^3} |\curl A + h_\ex \curl(A^0_\ep - A_{0,\ex})|^2 &= \int_{\R^3} |\curl A|^2 + h_\ex^2 |\curl(A^0_\ep - A_{0,\ex})|^2 \\
            &\hspace*{3.5cm} + 2h_\ex \int_{\R^3}\curl A \cdot \curl(A^0_\ep - A_{0,\ex})\\
            &\stackrel{\eqref{weak euler lagrange b0}}{=} \int_{\R^3}|\curl A|^2 + h_\ex^2 |\curl (A^0_\ep-A_{0,\ex})|^2 \\
            &\hspace*{3.5cm}- 2 h_\ex \int_\Omega \curl B^0_\ep \cdot A.
        \end{align*}
        Since $B^0_\ep \times \nu = 0$ on $\partial \Omega$, by integration by parts, we have
        \begin{equation}\label{b0 int by parts}
            \int_\Omega (\ip{iu}{\nabla_A u} + A) \cdot \curl B^0_\ep = \int_\Omega \mu(u,A) \cdot B^0_\ep.
        \end{equation}
        Combining our previous computations and inserting into \eqref{gl after lm}, we then deduce that
        \begin{multline}\label{split midstep}
            GL_\ep(\u,\A) = 
            E_\ep(\pmin) + F_{\ep,\pmin}(u,A) + \frac{1}{2}\int_{\R^3 \setminus \Omega} |\curl A|^2 
            - h_\ex \int_\Omega \mu(u,A) \cdot B^0_\ep \\
            + \frac{h_\ex^2}{2} \int_\Omega \left( \pmin^2 |\nabla \phi_\ep^0 - A^0_\ep|^2|u|^2 + \int_{\R^3} |\curl(A^0_\ep - A_{0,\ex})|^2  \right).
        \end{multline}
        Finally, recalling that
        \begin{align*}
            GL_\ep(\pmin e^{i h_\ex \phi_\ep}, h_\ex A^0_\ep) = E_\ep(\pmin) + \frac{h_\ex^2}{2}\left(\int_\Omega \pmin^2 |\nabla \phi_\ep^0 - A^0_\ep|^2 + \int_{\R^3} |\curl(A^0_\ep - A_{0,\ex)}|^2 \right),
        \end{align*}
        we conclude the proof by writing $A^0_\ep - \nabla \phi_\ep$ as $\dfrac{\curl B^0_\ep}{\pmin^2}$ and $|u|^2$ as $1+(|u|^2-1)$ in \eqref{split midstep}, which yields
        \begin{multline*}
            GL_\ep(\u,\A) = GL_\ep(\pmin e^{i h_\ex \phi_\ep}, h_\ex A^0_\ep) + \fen(u,A) + \frac{1}{2} \int_{\R^3 \setminus \Omega} |\curl A|^2\\ - h_\ex \int_\Omega \mu(u,A) \cdot B^0_\ep 
            + \frac{h_\ex^2}{2} \int_\Omega \frac{|\curl B^0_\ep|^2}{\pmin^2} (|u|^2-1).
        \end{multline*}
    \end{proof}
\begin{remark}\label{remark choice of phi}
Recalling \eqref{phi rho equation}, we emphasize that the specific choice of $\phipmin{A}^0$ is what guarantees the validity of \eqref{b0 existence}. This is a key ingredient when integrating parts in \eqref{b0 int by parts}, as it enables the vorticity term to emerge in our energy decomposition. The function $\phipmin{A}^0$ thus serves the same purpose as the function $\phi_0$ in the splitting result of \cite{Rom-CMP}*{Proposition 3.1}, which pertains to the homogeneous case.
\end{remark}

\section{Weighted isoflux problem}\label{section isoflux}
    This section is devoted to the proof of Proposition \ref{prop:conditionliminf}. To establish the result, we begin by demonstrating a key property of the weighted isoflux problem (see Definition \ref{def weighted isoflux}), which serves as a foundational component of our argument.

    \medskip
    First, let us introduce a family of currents induced by differential forms. Any $2$-form $\alpha$ supported in $\overline{\Omega}$ induces a $1$-current $[\alpha]$, defined by its action on $1$-forms $\omega$ as
    \begin{equation*}
        \ip{[\alpha]}{\omega} \colonequals \int_\Omega \alpha \wedge \omega.
    \end{equation*}
    For $[\alpha]$ to be an admissible current for the isoflux problem---that is, $[\alpha] \in \NN$---it is necessary and sufficient that $\alpha$ be a closed $2$-form, which is formalized in the following proposition.
    \begin{prop}\label{current of 2form}
        Let $\alpha$ is a smooth, continuous up to the boundary and supported in $\overline{\Omega}$. We have $[\alpha] \in \mathcal{N}(\Omega)$ if and only if $d\alpha = 0$.
        Furthermore, its mass is equal to
        \begin{equation}\label{2form weight}
            |[\alpha]| = \norm{\alpha}{L^1(\Omega)},
        \end{equation}
        and, in the case $[\alpha] \in \NN$, its boundary $\partial[\alpha]$ is the $0$-current, defined by its action on $0$-forms $f$ as
        \begin{equation*}
            \ip{\partial [\alpha]}{f} = \int_{\partial \Omega} f\alpha.
        \end{equation*}
    \end{prop}

    \begin{proof} We begin by proving that $[\alpha]$ has finite mass. Let $\mathcal{D}^1(\Omega)$ be the space of smooth, compactly supported $1$-forms in $\Omega$. By the dual characterization of $L^p$ norms, we have 
            \begin{align*}
                |[\alpha]| &= \sup_{\omega \in \mathcal{D}^1(\Omega), \norm{\omega}{L^\infty(\Omega)} = 1} \ip{[\alpha]}{\omega} \\
                &= \sup_{\omega \in \mathcal{D}^1(\Omega), \norm{\omega}{L^\infty(\Omega)} = 1}
                 \int_\Omega \alpha \wedge \omega \\
                &= \norm{\alpha}{L^1(\Omega)}<+\infty.
            \end{align*}
    Note that $\alpha$ is integrable, as it is continuous in $\overline{\Omega}$.

            \medskip 
            We now consider $\partial[\alpha]$. Let $f$ be a smooth $0$-form, that is, $f$ is a smooth function. In this context, the wedge product reduces to ordinary multiplication. Therefore, by Stokes' theorem, we have 
            $$
                \ip{\partial[\alpha]}{f} = \ip{[\alpha]}{df} 
                = \int_\Omega \alpha \wedge df 
                = \int_\Omega d(f\alpha) - fd\alpha
                = \int_{\partial \Omega} f\alpha - \int_\Omega fd\alpha.
            $$
            If $d\alpha = 0$, the current reduces to
            $$\ip{\partial[\alpha]}{f} = \int_{\partial \Omega} f\alpha.$$
            This means that $\partial[\alpha]$ is supported on $\partial \Omega$. 
            In addition, by arguing exactly as when computing $|[\alpha]|$, we find
            $$|\partial [\alpha]| = \norm{\alpha}{L^1(\partial \Omega)} < +\infty.$$
        We have thus proven that $[\alpha] \in \mathcal{N}(\Omega)$.
        On the other hand, if $[\alpha] \in \NN$, we then have that $\partial[\alpha]$ is supported in $\partial \Omega$. Therefore, for any $\phi \in C^\infty_0(\Omega)$, we have
            $$\ip{\partial [\alpha]}{\phi} = -\int_\Omega \phi d\alpha = 0.$$
        This implies that $d\alpha = 0$.
    \end{proof}

    We can now state a lower bound for $\max_{\Gamma \in \mathcal{N}(\Omega)} \RR_\eta(\Gamma)$ in the special case when $B$ is a solenoidal vector field (recall Definition \ref{def weighted isoflux}).
    \begin{prop}
        Suppose $\diver B = 0$ in $\Omega$. Then, we have
        \begin{equation}\label{isoflux lower bound}
            \max_{\Gamma \in \mathcal{N}(\Omega)} \RR_\eta(\Gamma) \geq \frac{\norm{B}{L^2(\Omega,\R^3)}}{\norm{\eta}{L^2(\Omega)}}.
        \end{equation}
    \end{prop}

    \begin{proof}
        Writing $B=B_xdx+B_ydy+B_zdz$, we have that $\star B$ is the $2$-form $B_x dy \wedge dz + B_y dz \wedge dx + B_z dx \wedge dy$, where $\star$ denotes the Hodge star operator. Moreover, $\star B$ is closed since $d (\star B) = (\diver B) dx \wedge dy \wedge dz = 0$. Thus, by Proposition \ref{current of 2form}, $[\star B] \in \mathcal{N}(\Omega)$ and $|[\star B]| = \norm{B}{L^1(\Omega,\R^3)}$. Moreover, it is not hard to see that $\eta[\star B]$ coincides with the current induced by the product $ [\eta \star B]$. This implies by \eqref{2form weight} that $|\eta [\star B]| = \norm{\eta B}{L^1(\Omega,\R^3)}$. Additionally, we have $\star B \wedge B = (|B_x|^2+|B_y|^2+|B_z|^2)dx\wedge dy \wedge dz$. By using $[\star B]$ as a competitor for the weighted isoflux problem, we have
        \begin{align*}
            \max_{\Gamma \in \mathcal{N}} \RR_\eta(\Gamma) \geq \RR_\eta([\star B]) 
            = \frac{\ip{[\star B]}{B}}{|\eta[\star B]|} 
            \stackrel{\eqref{2form weight}}{=} \frac{1}{\norm{\eta B}{L^1(\Omega,\R^3)}} \int_\Omega \star B \wedge B 
             &\geq \frac{\norm{B}{L^2(\Omega,\R^3)}^2}{\norm{\eta}{L^2(\Omega)} \norm{B}{L^2(\Omega,\R^3)}}\\
             &= \frac{\norm{B}{L^2(\Omega,\R^3)}}{\norm{\eta}{L^2(\Omega)}}.
        \end{align*}
    \end{proof}
    
    \subsection{Proof of Proposition~\ref{prop:conditionliminf}}
    We have all the ingredients to provide a proof for Proposition \ref{prop:conditionliminf}, which we recall applies to the special case $B=B_\ep^0$ and $\eta=\pmin^2$; see \eqref{def RRep}.    
    
    \begin{proof}[Proof of Proposition~\ref{prop:conditionliminf}]
        Assume $\liminf_{\ep \to 0^+} \RR_\ep= 0$. Since $\diver B_\ep^0=0$ in $\Omega$ and \eqref{rho uniform bound}, from \eqref{isoflux lower bound} it follows that, by passing to a subsequence if necessary, $B^0_\ep \to 0$ in $L^2(\Omega,\R^3)$. On the other hand, by taking $V=\Phi$ in \eqref{b0 integral equation}, where $\Phi$ is an arbitrary (but fixed in $\ep$) vector field in $C^\infty_0(\Omega,\R^3)$ with $\diver \Phi = 0$, we have, by integration by parts, that
        \begin{align*}
            \left|\int_\Omega H_{0,\ex} \cdot \Phi\right| = \left|\int_\Omega \frac{\curl B^0_\ep \cdot \curl \Phi}{\pmin^2} + B^0_\ep \cdot \Phi \right| 
            &\leq C \left|\int_\Omega \curl B^0_\ep \cdot \curl \Phi + B^0_\ep \cdot \Phi\right| \\
            &\leq C \left|\int_\Omega B^0_\ep \cdot \curl \curl \Phi + B^0_\ep \cdot \Phi\right|\\
            &\leq C \norm{B^0_\ep}{L^2(\Omega,\R^3)}\norm{\Phi}{H^2(\Omega,\R^3)},
        \end{align*}
        where we used an integration by parts. Note that, as $\ep \to 0$, the RHS converges to zero, while the LHS remains unchanged. This implies that 
        $$\int_\Omega H_{0,\ex} \cdot \Phi = 0, \quad \mbox{for any }\Phi \in C^\infty_0(\Omega,\R^3) \mbox{ such that }\diver \Phi = 0.
        $$ 
        Therefore, using \eqref{prop lsol characterizations dense}, we conclude that $\int_\Omega H_{0,\ex} \cdot A = 0$ for any $A \in \lsol^2(\Omega,\R^3)$, which, by orthogonality, implies that $H_{0,\ex}$ can be written as $\nabla \zeta$ for some $\zeta \in H^1(\Omega)$.This means that $\curl H_{0,\ex} \equiv 0$ in $\Omega$.
    \end{proof}
    
\section{Proof of the Main Result}\label{section main result}
\begin{proof}[Proof of Theorem \ref{main result lower bound}]
    We divide the proof into four steps. The first three focus on establishing items \ref{item weakvort} and \ref{item meisapprox}, while the final step addresses the analysis beyond the first critical field.
\begin{enumerate}[label=\textsc{\bf Step \arabic*.},leftmargin=0pt,labelsep=*,itemindent=*,itemsep=10pt,topsep=10pt]    
    \item {\bf Proving that $\fen(u,A) = O(h_\ex^2)$}. Since we are assuming that $(\u,\A)$ minimizes $GL_\ep$, we have 
    \begin{equation}\label{hypmeissner}
    GL_\ep(\u,\A) \leq GL_\ep(\meisconf).
    \end{equation}
    This implies that 
    \begin{equation}\label{free energy upper bound}
        \fen(u,A)+\frac{1}{2}\int_{\R^3 \setminus \Omega} |\curl A|^2 \stackrel{\eqref{energy splitting equation}}{\leq} h_\ex \int_\Omega  \mu(u,A) \cdot B^0_\ep + \err.
    \end{equation}
    We note that, by integration by parts,
    $$
    \int_\Omega  \mu(u,A) \cdot B^0_\ep=\int_\Omega \ip{iu}{\nabla_A u} \cdot \curl B^0_\ep +\int_\Omega \curl A \cdot B^0_\ep
    $$
    Therefore, using Hölder's inequality, we are led to
    \begin{align}\label{eqqqqq}
        \notag \fen(u,A) &\leq h_\ex \int_\Omega (\ip{iu}{\nabla_A u} \cdot \curl B^0_\ep + \curl A \cdot B^0_\ep) + \err \\
        &\notag\leq C\Bigg(h_\ex \norm{u}{L^4(\Omega,\C)} \norm{\nabla_A u}{L^2(\Omega,\C^3)} \norm{\curl B^0_\ep}{L^4(\Omega,\R^3)} \\
        &\notag\hspace*{7cm}+ \norm{\curl A}{L^2(\Omega,\R^3)} \norm{B^0_\ep}{L^2(\Omega,\R^3)} + |\err|\Bigg) \\
        &\stackrel{\eqref{b0 c01 regularity}}{\leq} C(h_\ex \norm{u}{L^4(\Omega,\C)} \norm{\nabla_A u}{L^2(\Omega,\C^3)} + \norm{\curl A}{L^2(\Omega,\R^3)} + |\err|).
    \end{align}
    Let us now individually bound from above the terms on the RHS. For $\norm{u}{L^4(\Omega,\C)}$, we use convexity to obtain
        $$
            \norm{u}{L^4(\Omega,\C)}^4 = \int_\Omega (1+(|u|^2-1))^2 \leq 2 \int_\Omega 1+(1-|u|^2)^2\leq C(1+\ep^2 \fen(u,A)),
        $$
        and therefore
        \begin{equation}\label{u l2 bound}
            \norm{u}{L^4(\Omega,\C)} \leq 
            C(1+\ep^{\frac{1}{2}}\fen(u,A)^{\frac14}).
        \end{equation}
        In the case of $\norm{\nabla_A u}{L^2(\Omega,\C^3)}$ and $\norm{\curl A}{L^2(\Omega,\R^3)}$, we have
        \begin{equation}\label{nabla u l2 bound}
            \norm{\nabla_A u}{L^2(\Omega,\C^3)} +\norm{\curl A}{L^2(\Omega,\R^3)}\leq C \fen(u,A)^\frac12.
        \end{equation}
        Finally, for $|\err|$, we have that
        \begin{align}\label{R0 bound}
            \notag|\err| &\leq \frac{h_\ex^2}{2} \int_\Omega \frac{|\curl B^0_\ep|^2}{\pmin^2}|(|u|^2-1)| \\
            \notag&\stackrel{\eqref{rho uniform bound}}{\leq} C h_\ex^2 \norm{\curl B^0_\ep}{L^4(\Omega,\R^3)}^2 \norm{1-|u|^2}{L^2(\Omega,\C)}\\
            &\stackrel{\eqref{b0 c01 regularity}}{\leq} Ch_\ex^2 \ep \fen(u,A)^\frac12.
        \end{align}
    Thus, from \eqref{eqqqqq}, \eqref{u l2 bound}, \eqref{nabla u l2 bound}, and \eqref{R0 bound}, we conclude that
    \begin{equation*}
        \fen(u,A) = O\left(h_\ex \fen(u,A)^\frac12\right),
    \end{equation*}
    which implies that there exists $C>0$ such that, for any $\ep$ sufficiently small, we have 
    \begin{equation}\label{free energy hex bound}
        \fen(u,A) \leq C h_\ex^2.
    \end{equation}

    \item {\bf Vorticity estimate. Proof of item \ref{item weakvort}.} Since $h_\ex = O(|\log \ep|)$, using \eqref{free energy hex bound}, we have 
    \begin{equation}\label{free energy log bound}
        \fen(u,A) \leq C|\log \ep|^2
    \end{equation}
    for some $C>0$ not depending on $\ep$. This estimate allows for applying Theorem~\ref{epleveltools}, which is crucial in deriving our quantitative estimates on $\fen(u,A)$ and $\mu(u,A)$. Using \eqref{vorticity estimate}, we deduce that, for any $\gamma \in (0,1)$,
    \begin{align}\label{mu nu approximation}
        \notag\int_\Omega \mu(u,A) \cdot B^0_\ep &\leq \norm{\mu(u,A) - \nu_\ep}{(C^{0,\gamma}_T(\Omega,\R^3))^*}\norm{B^0_\ep}{C^{0,\gamma}_T(\Omega,\R^3)} + \ip{\nu_\ep}{B^0_\ep}\\
        &\stackrel{\eqref{b0 c01 regularity} \& \eqref{free energy log bound}}{\leq} \ip{\nu_\ep}{B^0_\ep} + C(|\log \ep|^{2-\gamma q}).
    \end{align}
    By inserting \eqref{weighted lower bound} and \eqref{mu nu approximation} into \eqref{free energy upper bound}, and using \eqref{R0 bound} to control $|\err|$, we deduce, by choosing a large enough $n$ in \eqref{weighted lower bound}, and hence a large enough $q$, that
    \begin{multline}\label{free energy vorticity ep level estimates}
        \fen(u,A) - h_\ex \int_\Omega \mu(u,A) \cdot B^0_\ep \geq \frac{1}{2}|\pmin^2 \nu_\ep|\left(|\log \ep| - C \log|\log \ep| \right) \\- h_\ex \ip{\nu_\ep}{B^0_\ep} + o(|\log \ep|^{-2}).
    \end{multline}
    On the other hand, by definition of $\RR_\ep$ (recall \eqref{def RRep}), we have
    \begin{equation*}
        \ip{\nu_\ep}{B^0_\ep} \leq |\pmin^2 \nu_\ep| \RR_\ep.
    \end{equation*}
    Inserting this into \eqref{free energy vorticity ep level estimates}, we find that
    \begin{multline*}
        \fen(u,A) - h_\ex \int_\Omega \mu(u,A) \cdot B^0_\ep
        \geq \frac{1}{2} |\pmin^2 \nu_\ep|\left(|\log \ep| - 2h_\ex \RR_\ep - C \log|\log \ep| \right) + o(|\log \ep|^{-2}).
    \end{multline*}
    Using that $h_\ex\leq H_{c_1}^\ep-K_0\log|\log\ep|$, we observe that the leading-order term in the RHS cancels out, allowing us to deduce
    \begin{equation*}
        \fen(u,A) - h_\ex \int_\Omega \mu(u,A) \cdot B^0_\ep \geq \frac{1}{2} |\pmin^2 \nu_\ep|(2K_0 \RR_\ep - C)\log|\log \ep| + o(|\log \ep|^{-2}).
    \end{equation*}
    Observe that the assumption $\liminf_{\ep \to 0^+}\RR_\ep > 0$ then allows us to choose $K_0$, independent of $\ep$, such that $2K_0 \RR_\ep - C > 1$, leading us to
    \begin{equation}\label{EEq1}
        \fen(u,A) - h_\ex \int_\Omega \mu(u,A) \cdot B^0_\ep \geq \frac{1}{2} |\pmin^2 \nu_\ep| \log |\log \ep| + o(|\log \ep|^{-2}).
    \end{equation}
    
    On the other hand, from \eqref{free energy upper bound} and \eqref{R0 bound}, we deduce that
    \begin{equation}\label{free energy vorticity upper bound}
        \fen(u,A) + \frac{1}{2} \int_{\R^3 \setminus \Omega} |\curl A|^2 - h_\ex \int_\Omega \mu(u,A) \cdot B^0_\ep \leq |\err| =O(\ep |\log \ep|^3).
    \end{equation}
    Hence, by combining \eqref{EEq1} with \eqref{free energy vorticity upper bound}, we find that
    \begin{equation*}
        |\pmin^2 \nu_\ep| = o(|\log \ep|^{-1}).
    \end{equation*}
    Moreover, using \eqref{rho uniform bound} we find
    \begin{equation*}
        |\nu_\ep| \leq \frac{1}{b} |\pmin^2 \nu_\ep| = o(|\log \ep|^{-1}),
    \end{equation*}
    and therefore $\|\nu_\ep\|_{(C^{0,\gamma}_T(\Omega,\R^3))^*}=o(|\log \ep|^{-1})$ for any $\gamma\in (0,1]$. From the vorticity estimate \eqref{vorticity estimate}, we then deduce that
    \begin{equation}\label{vorticity weak converges to 0}
    \|\mu(u,A)\|_{(C^{0,\gamma}_T(\Omega,\R^3))^*}=o(|\log \ep|^{-1})\quad \mbox{for any }\gamma\in(0,1],
    \end{equation}
    which finishes the proof of item \ref{item weakvort}.

    \item {\bf Meissner configuration approximation. Proof of item \ref{item meisapprox}.} Using \eqref{vorticity weak converges to 0}, from \eqref{free energy vorticity upper bound} we find
    \begin{equation}\label{free energy converges to 0}
        \fen(u,A) + \frac{1}{2}\int_{\R^3 \setminus \Omega}|\curl A|^2 = o(1).
    \end{equation}
    Finally, by inserting \eqref{free energy converges to 0}, \eqref{vorticity weak converges to 0}, and \eqref{R0 bound} into the energy splitting \eqref{energy splitting equation}, we conclude that
    \begin{equation}\label{meissner approximation}
        GL_\ep(\u,\A) = GL_\ep(\meisconf) + o(1),
    \end{equation}
    which means that item \ref{item meisapprox} is satisfied. 
    
    \item {\bf Energetically optimal vortexless configuration beyond $\critfield$.} Now suppose we have a configuration $(\u,\A)$ such that $GL_\ep(\u,\A) \leq GL_\ep(\meisconf)$, $|\u| > c > 0$ and $h_\ex \leq \ep^{-\alpha}$ for some $\alpha < \frac{1}{2}$. Note that we also have $|u| > c > 0$. By appealing to a vorticity estimate for configurations satisfying such a lower bound, as the one provided in \cite{Rom-CMP}*{Proposition A.1}, we obtain that there exists a constant $C=C(\Omega)>0$ such that
    \begin{equation*}
        \norm{\mu(u,A)}{(C^{0,1}_T(\Omega,\R^3))^*} \leq C \ep F_\ep(u,A),
    \end{equation*}
    where $F_\ep(u,A)$ is the homogeneous free energy functional (that is, $\fen(u,A)$ with $\pmin \equiv 1$). In addition, from \cite{roman-vortex-construction}*{Lemma 8.1}, we have
    $$
    \norm{\mu(u,A)}{(C^0(\Omega,\R^3))^*} \leq C F_\ep(u,A).
    $$
    Given $\gamma \in (0,1)$, we can then derive an estimate for $\norm{\mu(u,A)}{(C^{0,\gamma}_T(\Omega,\R^3))^*}$ by interpolating between the norms in $(C^{0,1}_T(\Omega,\R^3))^*$ and $(C^0(\Omega,\R^3))^*$. In fact, we have
    $$
    \norm{\mu(u,A)}{(C^{0,\gamma}_T(\Omega,\R^3)^*} \leq \norm{\mu(u,A)}{(C^{0,1}_T(\Omega,\R^3)^*}^{\gamma} \norm{\mu(u,A)}{(C^0(\Omega,\R^3))^*}^{1-\gamma}= C\ep^{\gamma} F_\ep(u,A).
    $$
    Using the uniform bounds for $\pmin$ \eqref{rho uniform bound}, we conclude that
    \begin{equation}\label{improved vorticity estimates}
        \norm{\mu(u,A)}{(C^{0,\gamma}_T(\Omega,\R^3)^*} \leq C \ep^\gamma \fen(u,A),
    \end{equation}
    where $C=C(\Omega,b)>0$.

    \medskip
    Since we assume that \eqref{hypmeissner} holds, \eqref{free energy upper bound} remains valid. Then
    \begin{align*}
        \fen(u,A) + \frac{1}{2}\int_{\R^3 \setminus \Omega}|\curl A|^2 &\stackrel{\eqref{improved vorticity estimates} \& \eqref{b0 holderreg} \& \eqref{R0 bound}}{\leq}  C \left(h_\ex \ep^\gamma \fen(u,A) + h_\ex^2 \ep \fen(u,A)^{\frac{1}{2}}\right) \\
        &\hspace*{1cm}\stackrel{\eqref{free energy hex bound}}{\leq} C\ep^\gamma h_\ex^2 \fen(u,A)^\frac{1}{2}.
    \end{align*}
    It follows that
    \begin{equation*}
        \left(\fen(u,A) + \frac{1}{2}\int_{\R^3 \setminus \Omega} |\curl A|^2 \right)^\frac{1}{2} \leq C\ep^\gamma h_\ex^2.
    \end{equation*}
    Since we are assuming $h_\ex \leq \ep^{-\alpha}$ for $\alpha < \frac{1}{2}$, we choose $\gamma \in (2\alpha,1)$ to obtain
    \begin{equation}\label{free energy over hc1 goes to 0}
        \fen(u,A) + \frac{1}{2}\int_{\R^3 \setminus \Omega} |\curl A|^2 = o(1).
    \end{equation}
    Inserting \eqref{improved vorticity estimates}, \eqref{free energy over hc1 goes to 0}, and \eqref{R0 bound} into the energy splitting \eqref{energy splitting equation}, yields that \eqref{meissner approximation} is valid when $h_\ex\leq \ep^{-\alpha}$ for $\alpha<\frac12$. 
    \end{enumerate}
    The proof is thus finished.
\end{proof}

\section{\texorpdfstring{$\ep$}{ε}-level estimates for a weighted Ginzburg--Landau functional}\label{section eplevel}
In this section, we provide the changes needed to adapt the proof of the $\ep$-level estimates for the homogeneous Ginzburg--Landau functional in \cite{roman-vortex-construction}*{Theorem 1.1}, to the weighted case. The role played by the standard free-energy within $\Omega$, that is,
$$
F_\ep(u,A)=\frac12 \int_\Omega |\nabla_A u|^2+\frac1{2\ep^2}(1-|u|^2)^2+|\curl A|^2,
$$
will be replaced by $\fen(u,A)$. A crucial observation for adapting the arguments is that, since $0<c\leq \rho_\ep\leq 1$, we have that
\begin{equation}\label{crucialcomparisonenergy}
c^4 F_\ep(u,A)\leq \fen(u,A)\leq F_\ep(u,A).
\end{equation}

In what follows, we use the notation
$$
\eed(u)\colonequals \frac{\pmin^2}2  |\nabla u|^2+\frac{\pmin^4}{4\ep^2}(1-|u|^2)^2\quad \mbox{and}\quad 
\eed(u,A)\colonequals \eed(u)+|\curl A|^2.
$$

The construction in \cite{roman-vortex-construction} is lengthy and technical. For the reader's convenience, we will outline the key elements of the construction, though we will focus primarily on highlighting the main changes in the proofs. Interested readers are encouraged to track these changes alongside \cite{roman-vortex-construction}.

\subsection{Choice of grid}
Let us fix an orthonormal basis $(e_1,e_2,e_3)$ of $\R^3$ and consider a grid $\GG=\GG(a,\de)$ given by the collection of closed cubes $\CC_i\subset \R^3$ of side-length $\de=\de(\ep)$. In the grid we use a system of coordinates with origin in $a \in \Omega$ and orthonormal directions $(e_1,e_2,e_3)$. 
From now on we denote by $\RRR_1$ (respectively $\RRR_2$) the union of all edges (respectively faces) of the cubes of the grid. We have the following lemma, which directly follows from \cite{roman-vortex-construction}*{Lemma 2.1} by replacing $F_\ep(u,A)$ by $\fen(u,A)$ in view of \eqref{crucialcomparisonenergy}.

\begin{lemma}[Choice of grid]\label{Lemma:Grid} 
	For any $\gamma\in(-1,1)$ there exist constants $c_0(\gamma),c_1(\gamma)>0$, $\de_0(\Omega)\in(0,1)$ such that, for any $\ep,\de>0$ satisfying 
	$$
	\ep^{\frac{1-\gamma}2}\leq c_0\quad \mathrm{and}\quad c_1\ep^{\frac{1-\gamma}4}\leq \de \leq \de_0,
	$$
	if $(u,A)\in H^1(\Omega,\C)\times H^1(\Omega,\R^3)$ is a configuration such that $\fen(u,A)\leq \ep^{-\gamma}$ then there exists $b_\ep\in \Omega$ such that the grid $\GG(b_\ep,\delta)$ satisfies
	\begin{subequations}\label{propGrid}
		\begin{equation}\label{prop1Grid}
			|u_\ep|>5/8\quad \mathrm{on}\ \RRR_1(\GG(b_\ep,\de))\cap \Omega,
		\end{equation}
		\begin{equation*}
			\int\limits_{\RRR_1(\GG(b_\ep,\delta))\cap \Omega}\een(u,A)d\H^1\leq  C \de^{-2}\fen(u,A),
		\end{equation*}
		\begin{equation*}
			\int\limits_{\RRR_2(\GG(b_\ep,\delta))\cap \Omega}\een(u,A)d\H^2\leq  C \de^{-1}\fen(u,A),
		\end{equation*}
	\end{subequations}
	where $C$ is a universal constant, and where hereafter $\H^d$ denotes the $d$-dimensional Hausdorff measure for $d\in \N$.
\end{lemma}

From now on we drop the cubes of the grid $\GG(b_\ep,\de)$ given by the previous lemma, whose intersection with $\partial\Omega$ is non-empty. We also define 
\begin{equation}
	\Theta\colonequals \Omega \setminus \cup_{\CC_l\in \GG} \CC_l\quad\mathrm{and}\quad \partial \GG \colonequals \partial \left(\cup_{\CC_l\in \GG} \CC_l \right)\label{unioncubes}.
\end{equation}
Observe that, in particular, $\partial \Theta=\partial \GG \cup \partial \Omega$.

We remark that $\GG(b_\ep,\de)$ carries a natural orientation. The boundary of every cube of the grid will be oriented accordingly to this orientation. Each time we refer to a face $\omega$ of a cube $\CC$, it will be considered to be oriented with the same orientation of $\partial\CC$. If we refer to a face $\omega\subset \partial\GG$, then the orientation used is the same of $\partial\GG$.

\subsection{Ball construction method on a surface for the weighted functional}\label{subsec:ballmethodsurf}
The primary modification required to adapt the $\ep$-level estimates from \cite{roman-vortex-construction} to the case of a weighted functional involves incorporating the weight when developing the ball construction method on a surface. 

In this construction we will not need to obtain an estimate with ball per ball precision, but rather globally estimating the square of the weight by its minimum on the whole surface. Nevertheless, this is not trivial, since the weight appears with two different powers in the energy density $\eed(u)$ and, when bounding the energy from below, we need to extract the leading order power (i.e. $\pmin^2$, since $0<\pmin\leq 1$).

Our main result in this section is the following modification of \cite{roman-vortex-construction}*{Proposition 3.1}. It is worth recalling that this proposition extends the construction made by Sandier \cite{sandier-surface-ball-construction} of the ball construction on a surface, when the number of vortices is bounded, and that follows the method of Jerrard \cite{jerrard}.
\begin{prop}\label{ball}
	Let $\tilde \Sigma$ be a complete oriented surface in $\R^3$ whose second fundamental form is bounded by $1$. Let $\Sigma$ be a bounded open subset of $\tilde \Sigma$. For any $m,M>0$ there exists $\ep_0(m,M)>0$ such that, for any $\ep\in (0,\ep_0)$, if $u_\ep\in H^1(\Sigma,\C)$ satisfies
	$$
	\een(u_\ep,\Sigma)\colonequals  \int_\Sigma\eed(u_\ep) \leq M|\log \ep|^m
	$$
	and
	\begin{equation*}
		|u_\ep(x)|\geq \frac12\quad \textnormal{if } \d(x,\partial \Sigma)<1, 
	\end{equation*}
	where $\d(\cdot,\cdot)$ denotes the distance function in $\tilde \Sigma$, then, letting $d$ be the winding number of $u_\ep/|u_\ep|:\partial\Sigma\to S^1$ and $\mathcal M_\ep=M|\log \ep|^m$ we have 
	\begin{equation*}
	\een(u_\ep,\Sigma)\geq \pi |d|\min_{\tilde \Sigma} \pmin^2\left(\log \frac1{\ep} -\log \mathcal M_\ep \right).
	\end{equation*}
\end{prop}

An immediate consequence of this result, is the following modification of \cite{roman-vortex-construction}*{Corollary 3.1}.

\begin{corollary}\label{ball2}
	Let $\tilde \Sigma$ be a complete oriented surface in $\R^3$ whose second fundamental form is bounded by $Q_\ep=Q|\log\ep|^q$, where $q,Q>0$ are given numbers. Let $\Sigma$ be a bounded open subset of $\tilde \Sigma$. For any $m,M>0$ there exists $\ep_0(m,q,M,Q)>0$ such that, for any $\ep\in (0,\ep_0)$, if $u_\ep\in H^1(\Sigma,\C)$ satisfies
	$$
	\een(u_\ep,\Sigma)\colonequals  \int_\Sigma\eed(u_\ep) \leq M|\log \ep|^m
	$$
	and
	\begin{equation*}
		|u_\ep(x)|\geq \frac12\quad \textnormal{if } \d(x,\partial \Sigma)<Q_\ep^{-1}, 
	\end{equation*}
	where $\d(\cdot,\cdot)$ denotes the distance function in $\tilde \Sigma$, then, letting $d$ be the winding number of $u_\ep/|u_\ep|:\partial\Sigma\to S^1$ and $\mathcal M_\ep=M|\log \ep|^m$ we have 
	$$
	E_\ep(u_\ep,\Sigma)\geq \pi |d|\min_{\tilde \Sigma} \pmin^2\left(\log \frac1{\ep} -\log \mathcal M_\ep Q_\ep \right).
	$$	
\end{corollary}

The proof of this result is a minor modification of the proof of \cite{roman-vortex-construction}*{Corollary 3.1}. We include it here for the convenience of the reader.
\begin{proof} 
	For $y\in \Sigma_\ep\colonequals Q_\ep \Sigma$, we define
	$$
	\tilde u_\ep(y)=u_\ep \left(\frac{y}{Q_\ep} \right)\quad \mbox{and}\quad  \tilde \rho_\ep(y)=\pmin\left(\frac{y}{Q_\ep}\right).
	$$
	We let $\tilde \Sigma_\ep\colonequals Q_\ep \tilde \Sigma$.
	Observe that, by a change of variables, we have
	$$
	\een(u_\ep,\Sigma)=E_{\tilde \ep,\tilde \rho_\ep}(\tilde u_\ep,\Sigma_\ep),
	$$
	where $\tilde \ep \colonequals \ep Q_\ep$. It is easy to check that the second fundamental form of $\tilde \Sigma_\ep$ is bounded by $1$. Then a direct application of Proposition \ref{ball} shows that
	\begin{multline*}
	\een(u_\ep,\Sigma)=E_{\tilde\ep,\tilde \rho_\ep}(\tilde u_\ep,\Sigma_\ep)\geq \pi |d|\min_{\tilde \Sigma_\ep}\tilde \rho_\ep^2 \left( \log \frac1{\tilde \ep} -\log \mathcal M_\ep \right)\\
	=\pi |d|\min_{\tilde\Sigma}\pmin^2\left( \log \frac1\ep -\log \mathcal M_\ep Q_\ep \right)
	\end{multline*}
	for any $0<\ep<\ep_1=\ep_0Q_\ep^{-1}$, where $\ep_0$ is the constant appearing in the proposition.
\end{proof}

Let us now focus on proving Proposition \ref{ball}.
In what follows, the sets $B(x,t)$ and $S(x,t)$ denote respectively a geodesic (closed) ball and a geodesic sphere of radius $t$ centered at $x \in \Tilde{\Sigma}$. 
We will also use $S_e(x,t)$ and $B_e(x,t)$ to denote, respectively, the Euclidean sphere and the Euclidean ball of radius $t$ centered at $x$ in the tangent space $T_x \Sigma$.

The following result is the main new ingredient of our ball construction. It is analogous to \cite{roman-vortex-construction}*{Lemma 3.3} and \cite{sandier-surface-ball-construction}*{Lemma 3.12} when $\pmin\equiv 1$.
\begin{lemma}
	Suppose $u,\Sigma$ satisfy the hypotheses of Proposition \ref{ball}. There exist $\ep_0, r_0, C > 0$ depending only on $\Sigma$ such that, for any $\ep < \ep_0$, any $x \in \Sigma$, and any $t>0$ such that $\frac{\ep}{\min_{B(x,t)} \pmin^2} < t< r_0$, if $|u| \leq 1$ in $S(x,t)$, we have
	\begin{equation*}
		\erho{S(x,t)} \geq \min_{\Tilde{\Sigma}}\pmin^2\left(\pi m^2 \left(\frac{|d|}{t}-C \right)^+ + \frac{1}{C\ep}(1-m)^2 \right),
	\end{equation*}
    where $m \colonequals \inf_{S(x,t)} |u|$, and $$d\colonequals \begin{cases}
        \deg\left(\dfrac{u}{|u|},S(x,t)\right) & \text{ if } m\neq 0 \\
        0 & \text{otherwise}.
    \end{cases}$$
\end{lemma}

\begin{proof}
    This result is an extension of \cite{diaz-roman-2d}*{Lemma A.1}, which states that in $\Omega \subset \R^2$, for any $u \in H^1(\Omega)$ and any (Euclidean) ball $B_e(x,t) \subseteq \Omega$, we have
    \begin{equation}\label{euclidean ball estimate}
        \frac{1}{2}\int_{S_e(x,t)} \pmin^2 |\nabla |u||^2 + \frac{\pmin^4}{2\ep^2}(1-|u|^2)^2 \geq  \frac{\min_{B_e(x,t)} \pmin^2}{C\ep} (1-m)^2,
    \end{equation}
    where $C>0$ is a universal constant.
    
    \medskip
	We extrapolate this result to surfaces in $\R^3$ using the exponential map defined locally in $\Tilde{\Sigma}$. Fixing $u \in H^1(\Sigma,\C)$, $x \in \Sigma$, $t>0$ small enough, and using \eqref{euclidean ball estimate} with $\Omega = \exp_x^{-1}(\Sigma) \subseteq T_x \Sigma$, $\Tilde{u} \colonequals u \circ \exp_x$, and $\Tilde{\pmin} \colonequals \pmin \circ \exp_x$, yields
	\begin{equation}\label{lower bound on tansp}
		\frac{1}{2}\int_{S_e(x,t)} \Tilde{\pmin}^2 |\nabla |\Tilde{u}||^2 + \frac{\Tilde{\pmin}^4}{2\ep^2}(1-|\Tilde{u}|^2)^2 \geq  \frac{\min_{B_e(x,t)} \Tilde{\pmin}^2}{C\ep}(1-\Tilde{m})^2,
	\end{equation}
    where $\Tilde{m} \colonequals \inf_{S(x,t)}|\Tilde{u}|$.
	The bound on the second fundamental form of $\Tilde{\Sigma}$ implies that there exists $r_0>0$ depending only on $\Sigma$ such that, for $t<r_0$, $\exp_x$ is a quasi-isometry in $B_e(x,t)$ and $\exp_x(S_e(x,t))$ is a geodesic sphere in $\Sigma$ centered at $x$. By performing a change of variables on \eqref{lower bound on tansp}, we obtain
	\begin{equation}\label{lower bound on surface}
		\frac{1}{2}\int_{S(x,t)} \rho_\ep^2 |\nabla |u||^2 + \frac{\pmin^4}{2\ep^2}(1-|u|^2)^2 \geq \frac{\min_{B(x,t)} \rho_\ep^2}{C\ep}(1-m)^2 \geq  \frac{\min_{\Tilde{\Sigma}} \rho_\ep^2}{C\ep}(1-m)^2, 
	\end{equation}
	with potentially different values for the constants $r_0$ and $C$, after the change of coordinates. This yields the lower bound
	\begin{align}\label{geodesic sphere energy lower bound}
		\erho{S(x,t)}
		&= \frac{1}{2} \int_{S(x,t)} \rho_\ep^2\left(|\nabla |u||^2 + |u|^2 \left|\nabla \frac{u}{|u|} \right|^2 \right) + \frac{\pmin^4}{2\ep^2}(1-|u|^2)^2 \notag\\
		&\stackrel{\eqref{lower bound on surface}}{\geq} \min_{\Tilde{\Sigma}} \pmin^2\left(\frac{1}{C \ep}(1-m)^2 + m^2 \int_{S(x,t)} \left|\nabla \frac{u}{|u|}\right|^2 \right).
	\end{align}
    For the second term in the RHS of \eqref{geodesic sphere energy lower bound}, we use the following result used in the proof of \cite{sandier-surface-ball-construction}*{Lemma 3.12}, which states that for small enough $t$,
	\begin{equation}\label{energy on sphere degree lower bound}
		\int_{S(x,t)} \left|\nabla \frac{u}{|u|}\right|^2 \geq \pi \left(\frac{|d|}{t}- C \right)^+.
	\end{equation}
	By inserting \eqref{energy on sphere degree lower bound} into \eqref{geodesic sphere energy lower bound}, we complete the proof.
\end{proof}
With this modification at hand, the proof of Proposition~\ref{ball} follows almost verbatim \cite{roman-vortex-construction}*{Section 3.1}, the only difference being that the lower bound $\min_{\Tilde{\Sigma}} \pmin^2$ is carried through the subsequent steps.

\subsection{2D vorticity estimate}
Let $\omega$ be a two dimensional domain. For a given function $u:\overline \omega\to \C$ and a given vector field $A:\overline \omega\to \R^2$ we let
$$
\fen(u,A,\omega)= \int_\omega \eed(u,A),\quad \fen(u,A,\partial \omega)=\int_{\partial \omega}\eed(u,A)d\H^1. 
$$
We also denote $F(u,A,\omega)=F_{\ep,1}(u,A,\omega)$ and $F(u,A,\partial \omega)=F_{\ep,1}(u,A,\partial \omega)$. Analog to \eqref{crucialcomparisonenergy}, we have
$$
c^4 F_\ep(u,A,\omega)\leq  \fen(u,A,\omega)\leq F_\ep(u,A,\omega)
$$
and
$$
c^4 F_\ep(u,A,\partial \omega)\leq \fen(u,A,\partial\omega)\leq F_\ep(u,A,\partial\omega).
$$
Hence, by combining these inequalities with \cite{roman-vortex-construction}*{Theorem 4.1}, we get the following 2D vorticity estimate.
\begin{prop}\label{Teo:2dEstimate} Let $\omega\subset \R^2$ be a bounded domain with Lipschitz boundary. Let $u:\omega \rightarrow \C$ 
	and $A:\omega\rightarrow \R^2$ be $C^1(\overline \omega)$ and such that $|u|\geq 5/8$ on $\partial \omega$. 
	Let $\{U_j\}_{j\in J}$ be the collection of connected components of $\{|1-|u(x)||\geq 1/2\}$ and $\{S_i\}_{i\in I}$ denote the collection of connected components of $\{|u(x)|\leq 1/2\}$ whose degree $d_i=\mathrm{deg}(u/|u|,\partial S_i)\neq 0$.
	Then, letting $r=\sum_{j\in J} \mathrm{diam}(U_j)$ and assuming $\ep,r\leq 1$, we have
	\begin{equation*}
		\left\| \mu(u,A) -2\pi \sum_{i\in I}d_i\delta_{a_i} \right\|_{C^{0,1}(\omega)^*}\leq C\max(\ep,r)(1+\fen(u,A,\omega)+\fen(u,A,\partial\omega)),
	\end{equation*}
	where $a_i$ is the centroid of $S_i$ and $C$ is a universal constant.
\end{prop}

Given a two dimensional Lipschitz domain $\omega \subset \Omega$, we let $(s,t,0)$ denote coordinates in $\R^3$ such that $\omega\subset \{(s,t,0) \in \Omega\}$. We define $\mu_\ep\colonequals \mu_\ep(u,A)[\partial_s,\partial_t]$, and write $\mu_{\ep,\omega}$ its restriction to $\omega$. From the previous proposition we immediately get the following result, which is analogous to \cite{roman-vortex-construction}*{Corollary 4.1}.
\begin{corollary}\label{cor:2dVortEstimate}
	Let $\gamma\in(0,1)$ and assume that $(u,A)\in H^1(\Omega,\C)\times H^1(\Omega,\R^3)$ is a configuration such that $\fen(u,A)\leq \ep^{-\gamma}$, so that by Lemma \ref{Lemma:Grid} there exists a grid $\GG(b_\ep,\de)$ satisfying \eqref{prop1Grid}. Then there exists $\ep_0(\gamma)$ such that, for any $\ep<\ep_0$ and for any face $\omega\subset \RRR_2(\GG(b_\ep,\de))$ of a cube of the grid $\GG(b_\ep,\de)$, letting $\{U_{j,\omega}\}_{j\in J_\omega}$ be the collection of connected components of $\{x\in\omega \ | \ |1-|u(x)||\geq 1/2\}$ and
	$\{S_{i,\omega}\}_{i\in I_\omega}$ denote the collection of connected components of $\{x\in\omega \ | \ |u_\ep(x)|\leq 1/2\}$ whose degree $d_{i,\omega}\colonequals \mathrm{deg}(u/|u|,\partial S_{i,\omega})\neq 0$, we have
	\begin{multline*}
		\left\| \mu_{\ep,\omega} -2\pi \sum_{i\in I_\omega}d_{i,\omega}\delta_{a_{i,\omega}} \right\|_{C^{0,1}(\omega)^*}\leq \\ C\max(r_\omega,\ep)\left(1+\int_\omega \eed(u,A)d\H^2+\int_{\partial \omega} \eed(u,A)d\H^1\right),
	\end{multline*}
	where $a_{i,\omega}$ is the centroid of $S_{i,\omega}$, $r_\omega\colonequals \sum_{j\in J_\omega}\mathrm{diam}(U_{j,\omega})$, and $C$ is a universal constant.
\end{corollary}

\subsection{Construction of the vorticity approximation\label{sec:constrvortapprox}}
The construction of the vorticity approximation given in \cite{roman-vortex-construction}*{Section 5} is exactly the same we need here. What changes when using $\fen(u_\ep,A_\ep)$ instead of $F(u_\ep,A_\ep)$ is the location of the set $\left\{|u_\ep|\leq 1/2  \right\}$, which is where the vortex lines are located. This is detected through the new version of the vorticity estimate given in Corollary \ref{cor:2dVortEstimate}. Even though our intention is not to repeat the whole construction here, let us next mention a few  of the main ingredients for the convenience of the reader. 

Let $\gamma\in (0,1)$ and a pair $(u,A)\in H^1(\Omega,\C)\times H^1(\Omega,\R^3)$ such that $\fen(u,A)\leq \ep^{-\gamma}$. We can then apply Lemma \ref{Lemma:Grid}, which provides a grid of cubes $\GG(b_\ep,\de)$ satisfying \eqref{propGrid}. On the face $\omega$ of each cube $\CC_l$ of the grid, Corollary \ref{cor:2dVortEstimate} provides the existence of points $a_{i,\omega}$ and integers $d_{i,\omega}\neq 0$ such that
\begin{equation}\label{apmu}
	\mu_{\ep,\omega}\approx 2\pi \sum_{i\in I_\omega} d_{i,\omega}\delta_{a_{i,\omega}}.
\end{equation}
Our vorticity approximation in $\CC_l$, which we denote $\nu_{\ep,\CC_l}$, is then constructed in such a way that its restriction to the face $\omega$ of any cube of the grid coincides with the right-hand side of \eqref{apmu}. A crucial observation is that since $\partial\mu(u,A)=0$ relative to any cube of the grid $\CC_l$, we have
$$
\sum_{\omega\subset \partial \CC_l} \sum_{i\in I_\omega} d_{i,\omega}=0,
$$
which allows to define the vorticity approximation in any cube of the grid as the minimal connection, as first introduced in \cite{BreCorLie}, associated to the collection of positive and negative points (according to their degree) where the vortices are located on the boundary of the cube. 

Analogously, since $\partial \mu (u,A)=0$ relative to $\partial \GG$, we have 
$$
\sum_{\omega\subset \partial \GG} \sum_{i\in I_\omega} d_{i,\omega}=0,
$$
which allows to define the vorticity approximation in $\Theta$ (recall \eqref{unioncubes}), which we denote $\nu_{\ep,\Theta}$, as the minimal connection through the boundary $\partial\Omega$ associated to the collection of positive and negative points where the vortices are located on $\partial\GG$. 

It is important to remark that by using \eqref{crucialcomparisonenergy} we directly obtain a control of the size of the set where $\nu_\ep$ is supported. In fact, by combining \eqref{crucialcomparisonenergy} with \cite{roman-vortex-construction}*{Lemma 5.5}, we obtain the following result.

\begin{lemma}\label{Lemma:support}
	Let $\gamma\in(0,1)$ and assume that $(u_\ep,A_\ep)\in H^1(\Omega,\C)\times H^1(\Omega,\R^3)$ is a configuration such that $\fen(u_\ep,A_\ep)\leq \ep^{-\gamma}$, so that, by Lemma \ref{Lemma:Grid}, there exists a grid $\GG(b_\ep,\de)$ satisfying \eqref{propGrid}.
	For each face $\omega\subset \RRR_2(\GG(b_\ep,\de))$ of a cube of the grid, let $|I_\omega|$ be the number of connected components of $\{x\in \omega \ | \ |u_\ep(x)|\leq 1/2 \}$ whose degree is different from zero. Then, letting
	\begin{equation}\label{CubesUsed}
		\GG_0\colonequals \{ \CC_l\in \GG \ | \ \textstyle \sum_{\omega\subset\partial \CC_l}|I_\omega| > 0\},
	\end{equation}
	we have
	$$
	\mathrm{supp}(\nu_\ep)\subset S_{\nu_\ep}\colonequals
	\bigcup_{\CC_l\in \GG_0} \CC_l \cup \left\{\begin{array}{cl}\overline\Theta&\mathrm{if}\ \sum_{\omega \subset \partial\GG}|I_\omega|>0\\ \emptyset&\mathrm{if}\ \sum_{\omega \subset \partial\GG}|I_\omega|=0\end{array} \right. .
	$$
	Moreover
	$$
	|S_{\nu_\ep}|\leq C\de (1+\de \fen(u_\ep,A_\ep)),
	$$
	where $C$ is a constant depending only on $\partial\Omega$.
\end{lemma}

\subsection{Lower bound for \texorpdfstring{$E_\ep(u_\ep)$}{E(u)} far from the boundary}
We now proceed to provide a lower bound for the energy without magnetic field $\een(u_\ep)$ in the union of cubes of the grid $\GG(b_\ep,\de)$. The proof is based on slicing according to the level sets of the smooth approximation of the function $\zeta$ constructed in \cite{roman-vortex-construction}*{Appendix A}. The function $\zeta$ that we use here is exactly the same considered in \cite{roman-vortex-construction}. It can be seen as a calibration of the minimal connection used when constructing the vorticity approximation. Its smooth approximation, coupled with quantitative estimates on the second fundamental form of its level sets, is one of the most technical parts of \cite{roman-vortex-construction}. 

The aforementioned slicing procedure is coupled with the ball construction method on a surface of Section \ref{subsec:ballmethodsurf}, where the quantitative estimate on the second fundamental form of the smooth approximation of the calibration function plays a crucial role. 

The following result is analogous to \cite{roman-vortex-construction}*{Proposition 6.1}.
\begin{prop}\label{propcubes}
	Let $m,M>0$ and assume that $(u_\ep,A_\ep)\in H^1(\Omega,\C)\times H^1(\Omega,\R^3)$ is such that $\fen(u_\ep,A_\ep)=M_\ep\leq M|\log \ep|^m$. For any $b,q>0$, there exists $\ep_0>0$ depending only on $b,q,m$, and $M$, such that, for any $\ep<\ep_0$, letting $\GG(b_\ep,\de)$ denote the grid given by Lemma \ref{Lemma:Grid} with $\de=\de(\ep)=|\log \ep|^{-q}$, if ---recall \eqref{CubesUsed}---
	\begin{equation*}
		\een(u_\ep,\cup_{\CC_l\in \GG_0} \CC_l)\leq KM_\ep\quad\mathrm{and}\quad \sum_{\CC_l\in\GG_0}\int_{\partial\CC_l} \eed(u_\ep)d\H^2\leq K\de^{-1}M_\ep
	\end{equation*}
	for some universal constant $K$, then
	$$
	\een(u_\ep,\cup_{\CC_l\in \GG_0} \CC_l)\geq \frac12\sum_{\CC_l\in \GG_0}\min_{\CC_l}\pmin^2|\nu_{\ep,\CC_l}|\left(\log \frac1\ep-\log C\frac{M_\ep^{56}|\log \ep|^{7(1+b)}}{\de^{55}}\right)-\frac C{|\log \ep |^b},
	$$
	where $C$ is a universal constant.
\end{prop}
By replacing \cite{roman-vortex-construction}*{Corollary 3.1} by Corollary \ref{ball2}, the proof of this result is a minor modification of the proof of \cite{roman-vortex-construction}*{Proposition 6.1}. On each cube, one integrates over slices $\Sigma$ of the smooth approximation of the calibration function, which makes appear the factor $\min_{\Sigma}\pmin^2$ on the lower bound. Since the level sets are contained in the cube, we have $\min_{\Sigma}\pmin^2\geq \min_{\CC_l} \pmin^2$, which can then be factored out, allowing the proof to continue exactly as in \cite{roman-vortex-construction}.

\subsection{Lower bound for \texorpdfstring{$E_\ep(u_\ep)$}{E(u)} close to the boundary}
We now proceed to provide a lower bound for the energy without magnetic field $\een(u_\ep)$ in $\Theta$ (recall \eqref{unioncubes}). The proof is based on slicing according to the level sets of the smooth approximation of the function $\zeta$ constructed in \cite{roman-vortex-construction}*{Appendix B}. The function $\zeta$ that we use here is exactly the same considered in \cite{roman-vortex-construction}. It can be seen as a calibration of the minimal connection through the boundary $\partial\Omega$ used when constructing the vorticity approximation. Once again, the slicing procedure is coupled with the ball construction method on a surface of Section \ref{subsec:ballmethodsurf}, where the quantitative estimate on the second fundamental form of the smooth approximation of the calibration function plays a crucial role. 

The following result is analogous to \cite{roman-vortex-construction}*{Proposition 7.1}.

\begin{prop}\label{propboundary}
	Suppose that $\partial\Omega$ is of class $C^2$.
	Let $m,M>0$ and assume that $(u_\ep,A_\ep)\in H^1(\Omega,\C)\times H^1(\Omega,\R^3)$ is such that $\fen(u_\ep,A_\ep)=M_\ep\leq M|\log \ep|^m$. For any $b,q>0$, there exists $\ep_0>0$ depending only on $b,q,m,M$, and $\partial\Omega$, such that, for any $\ep<\ep_0$, letting $\GG(b_\ep,\de)$ denote the grid given by Lemma \ref{Lemma:Grid} with $\de=\de(\ep)=|\log \ep|^{-q}$, if 
	\begin{equation*}
		\een(u_\ep,\Theta)\leq KM_\ep\quad \mathrm{and}\quad \int_{\partial\GG}\eed(u_\ep)d\H^2\leq K\de^{-1}M_\ep
	\end{equation*}
	for some universal constant $K>0$, then 
	$$
	\een(u_\ep,\Theta)\geq \frac12\min_{\overline \Theta}\pmin^2| \nu_{\ep,\Theta}|\left(\log \frac1\ep-\log C\frac{M_\ep^{124}|\log \ep|^{(20+1/3)(1+b)}}{\de^{123}}\right)-\frac{C}{|\log \ep |^{b}},
	$$
	where $C$ is a constant depending only on $\partial\Omega$.
\end{prop}
Once again, by replacing \cite{roman-vortex-construction}*{Corollary 3.1} by Corollary \ref{ball2}, the proof of this result is a minor modification of the proof of \cite{roman-vortex-construction}*{Proposition 7.1}. In $\Theta$, one integrates over slices $\Sigma$ of the smooth approximation of the calibration function, which makes appear the factor $\min_{\Sigma}\pmin^2$ on the lower bound. Since the level sets are contained in $\overline \Theta$, we have $\min_{\Sigma}\pmin^2\geq \min_{\overline \Theta} \pmin^2$, which can then be factored out, allowing the proof to continue exactly as in \cite{roman-vortex-construction}.

\subsection{Proof of the \texorpdfstring{$\ep$}{ε}-level estimates}
First, by replacing \cite{roman-vortex-construction}*{Theorem 4.1} with Proposition \ref{Teo:2dEstimate}, the proof of \eqref{vorticity estimate}, that is, the vorticity estimate part of the $\ep$-level estimates, is exactly the same as the one presented in \cite{roman-vortex-construction}*{Section 8.1}. Next, we present the proof of the lower bound \eqref{weighted lower bound}.
\begin{proof}
	By gauge invariance of the energy $\fen(u_\ep,A_\ep)$, without loss of generality we can assume that $A_\ep$ is in the Coulomb gauge, that is,
	$$
	\diver A_\ep=0 \ \mathrm{in}\ \Omega
	\quad \mathrm{and}\quad A_\ep\cdot \nu =0 \ \mathrm{on}\ \partial\Omega.
	$$
	It directly follows that
	$$
	\|A_\ep\|_{H^1(\Omega,\R^3)}\leq C\|\curl A_\ep\|_{L^2(\Omega, \R^3)},
	$$
	where throughout the proof $C>0$ denotes a constant that may change from line to line. By Sobolev embedding, we then have 
	$$
	\|A_\ep\|_{L^p(\Omega,\R^3)}\leq C\|A_\ep\|_{H^1(\Omega,\R^3)}
	$$
	for any $1\leq p\leq 6$.
	
	Observe that
	\begin{equation}\label{eqqq1}
	\frac12\int_\Omega \pmin^2|\nabla u_\ep|^2\leq \int_\Omega \pmin^2|\nabla_{A_\ep} u_\ep|^2 +\pmin^2|u_\ep|^2|A_\ep|^2 \leq \int_\Omega |\nabla_{A_\ep} u_\ep|^2+|u_\ep|^2|A_\ep|^2.
	\end{equation}
	Using H\"older's inequality, we find
	\begin{align*}
	\int_\Omega |u_\ep|^2|A_\ep|^2 &\leq \int_\Omega |A_\ep|^2+\int_\Omega (|u_\ep|^2-1)|A_\ep|^2\\
	&\leq \|A_\ep\|_{L^6(\Omega)}^2|\Omega|^\frac23+\||u_\ep|^2-1\|_{L^2(\Omega)}\|A_\ep\|_{L^6(\Omega)}^2|\Omega|^\frac16\\
	&\leq C\|\curl A_\ep\|_{L^2(\Omega)}^2(1+\ep F_\ep(u_\ep,A_\ep)^\frac12)\\
	&\leq C \fen(u_\ep,A_\ep)(1+\ep \fen(u_\ep,A_\ep)^\frac12).
	\end{align*}
	Combining this with \eqref{eqqq1}, we find
	$$
	\een(u_\ep)\leq C \fen(u_\ep,A_\ep).
	$$
    
	Now, we consider the grid $\GG(b_\ep,\de)$ given by Lemma \ref{Lemma:Grid}. It is easy to see that, up to an adjustment of the constant appearing in the lemma, we can require our grid to additionally satisfy the inequalities
	\begin{align*}
		\int_{\RRR_1(\GG(b_\ep,\de))}\eed(u_\ep)d\H^1&\leq C\de^{-2}\fen(u_\ep,A_\ep),\\ \int_{\RRR_2(\GG(b_\ep,\de))}\eed(u_\ep)d\H^2&\leq C\de^{-1}\fen(u_\ep,A_\ep).
	\end{align*}
	Let us now consider the vorticity approximation $\nu_\ep$ defined in Section~\ref{sec:constrvortapprox} together with the set $S_{\nu_\ep}$ (recall Lemma~\ref{Lemma:support}). Observe that
	\begin{align*}
		\int_{S_{\nu_\ep}} \pmin^2|\nabla u_\ep|^2&= \int_{S_{\nu_\ep}} \pmin^2|\nabla_{A_\ep} u_\ep|^2 +\pmin^2|u_\ep|^2|A_\ep|^2+2\pmin^2\langle \nabla_{A_\ep}u_\ep, 	i u_\ep A_\ep \rangle\\
		&\leq \int_{S_{\nu_\ep}} \pmin^2|\nabla_{A_\ep} u_\ep|^2+\pmin^2|u_\ep|^2|A_\ep|^2+2\pmin^2| \nabla_{A_\ep}u_\ep||u_\ep| |A_\ep|.
	\end{align*}
	Let $t_\ep>0$. Using that $2xy\leq t_\ep x^2+\frac1{t_\ep}y^2$, we have that
	$$
	\int_{S_{\nu_\ep}}2\pmin^2| \nabla_{A_\ep}u_\ep||u_\ep| |A_\ep|\leq t_\ep\int_{S_{\nu_\ep}}\pmin^2| \nabla_{A_\ep}u_\ep|^2+\frac{1}{t_\ep}\int_{S_{\nu_\ep}}\pmin^2 |u_\ep|^2 |A_\ep|^2.
	$$
	We then deduce that (recall that $\pmin^2\leq 1$)
	\begin{align*}
	\int_{S_{\nu_\ep}} \pmin^2|\nabla u_\ep|^2&\leq (1+t_\ep)\int_{S_{\nu_\ep}} \pmin^2|\nabla_{A_\ep} u_\ep|^2+\left(1+\frac1{t_\ep}\right)\int_{S_{\nu_\ep}}|u_\ep|^2|A_\ep|^2\\
	&=(1+t_\ep)\int_{S_{\nu_\ep}} \pmin^2|\nabla_{A_\ep} u_\ep|^2+\left(1+\frac1{t_\ep}\right)\left(\int_{S_{\nu_\ep}}(|u_\ep|^2-1)|A_\ep|^2+|A_\ep|^2\right).
	\end{align*}
	Using H\"older's inequality, we find
	$$
	\int_{S_{\nu_\ep}}(|u_\ep|^2-1)|A_\ep|^2\leq \||u_\ep|^2-1\|_{L^2(S_{\nu_\ep})}|S_{\nu_\ep}|^{\frac16}\|A_\ep\|^2_{L^6(S_{\nu_\ep},\R^3)}
	$$
	and
	$$
	\int_{S_{\nu_\ep}}|A_\ep|^2\leq |S_{\nu_\ep}|^{\frac23}\|A_\ep\|_{L^6(S_{\nu_\ep},\R^3)}^2.
	$$
	Hence, recalling \eqref{crucialcomparisonenergy}, we deduce that
	\begin{multline*}
	\int_{S_{\nu_\ep}}\pmin^2|\nabla u_\ep|^2	\leq \int_{S_{\nu_\ep}} \pmin^2|\nabla_{A_\ep} u_\ep|^2
	+t_\ep \fen(u_\ep,A_\ep)\\+C \left(1+\frac1{t_\ep}\right)   \fen(u_\ep,A_\ep)\left(\ep\fen(u_\ep,A_\ep)^\frac12|S_{\nu_\ep}|^{\frac16}+|S_{\nu_\ep}|^{\frac23}\right).
\end{multline*}
	We let $t_\ep=|\log\ep|^{-(m+n)}$ and choose $\delta=\delta(\ep)$ such that $\delta\leq |\log\ep|^{-3(m+n)}$ (the final choice of this parameter will come after). Note that, $t_\ep \fen (u_\ep,A_\ep)\leq C|\log\ep|^{-n}$, $|S_{\nu_\ep}|\leq C\de=C|\log\ep|^{-3(m+n)}$, and 
	$$
	\left(1+\frac1{t_\ep}\right)   \fen(u_\ep,A_\ep)\left(\ep \fen(u_\ep,A_\ep)^\frac12|S_{\nu_\ep}|^{\frac16}+|S_{\nu_\ep}|^{\frac23}\right)\leq C|\log\ep|^{-n}.
	$$
	In particular,
	\begin{equation}\label{energy2}
	\int_{S_{\nu_\ep}}\pmin^2|\nabla u_\ep|^2	\leq \int_{S_{\nu_\ep}} \pmin^2|\nabla_{A_\ep} u_\ep|^2+C|\log\ep|^{-n}.
	\end{equation}
	
	\medskip 
	We now apply Proposition \ref{propcubes} and Proposition \ref{propboundary} with $b=n$ and $q=3(m+n)$. This yields the existence of $\ep_0>0$, depending only on $m,n,M,$ and $\partial\Omega$, such that, for any $0<\ep<\ep_0$,
	$$
	E_\ep(u_\ep,S_{\nu_\ep})\geq \frac12\Lambda_\ep(\nu_\ep,\pmin)\left(\log \frac1\ep-\log C\frac{M_\ep^{124}|\log \ep|^{(20+1/3)(1+n)}}{\de^{123}}\right)-\frac{C}{|\log \ep |^{n}},
	$$
	where 
    \begin{equation}\label{Lambdaep}
    \Lambda_\ep(\nu_\ep,\pmin)=\min_{\overline \Theta}\pmin^2| \nu_{\ep,\Theta}|+\sum_{\CC_l\in \GG_0}\min_{\CC_l}\pmin^2|\nu_{\ep,\CC_l}|,
    \end{equation}
    $M_\ep=\fen(u_\ep,A_\ep)$, and $C$ is a constant depending only on $\partial\Omega$. 
	By combining this with \eqref{energy2}, we are led to 
	\begin{multline}\label{finalest1}
		\int_{S_{\nu_\ep}}\pmin^2|\nabla_{A_\ep}u_\ep|^2+\frac{\pmin^4}{2\ep^2}(1-|u_\ep|^2)^2+|\curl A_\ep|^2\\
		\geq
		\Lambda_\ep(\nu_\ep,\pmin)\left(\log \frac1\ep-\log C\frac{M_\ep^{124}|\log \ep|^{(20+1/3)(1+n)}}{\de^{123}}\right)-\frac{C}{|\log \ep |^{n}}.
	\end{multline}
	We have reached the point where the hypotheses on $\pmin$ are needed. On the one hand, since there exist $\alpha\in (0,1)$, $N>0$, and $C_1>0$ (that do not depend on $\ep$) such that
	$$
	\|\pmin\|_{C^{0,\alpha}(\Omega)}\leq C_1|\log\ep|^N,
	$$
	we deduce that, for any cube $\CC_l \in \GG_0$ and any $x\in \CC_l$, we have (recall that $\pmin\leq 1$)
	$$
	\left|\pmin^2(x)-\min_{\CC_l}\pmin^2\right|\leq 2\left|\pmin(x)-\min_{\CC_l}\pmin\right|\leq 2 C_1 \delta^\alpha |\log\ep|^N.
	$$
	It then follows that
	$$
	\left|\min_{\CC_l}\pmin^2|\nu_{\ep,\CC_l}|-|\pmin^2\nu_{\ep,\CC_l}| \right|\leq C \delta^\alpha |\log\ep|^N.
	$$
	By choosing $\delta=\delta(\ep)$ such that $\delta\leq |\log\ep|^{-\frac{n+N+1}{\alpha}}$, we conclude that 
	\begin{equation}\label{finalest2}
	\left|\min_{\CC_l}\pmin^2|\nu_{\ep,\CC_l}|-|\pmin^2\nu_{\ep,\CC_l}|\right||\log\ep|\leq C |\log\ep|^{-n}.
	\end{equation}
    
	On the other hand, since there exist $C_2>0$ and $\kappa>0$ (that do not depend on $\ep$) such that $a_\ep$ is constant in the set $\{ x\in \Omega \ | \ \dist(x,\partial\Omega)\leq C_2|\log\ep|^{-\kappa})\}$, from Proposition~\ref{prop:rho_ep like a_ep}, we deduce that $\pmin$ is almost constant in the same set (which we make precise below). Then, by choosing $\delta=\delta(\ep)$ such that $\delta\leq |\log\ep|^{-(\kappa+1)}$, we have that 
	$$
	\Theta \subset \{ x\in \Omega \ | \ \dist(x,\partial\Omega)\leq k(\ep)\},
	$$
	and therefore Proposition~\ref{prop:rho_ep like a_ep} implies that, for any $x\in \Theta$, we have
	$$
	\left|\pmin^2(x)-\min_{\overline{\Theta}}\pmin^2\right|\leq Ce^{-\frac{c}{\ep |\log\ep|^\kappa}}.
	$$ 
	Hence,
	\begin{equation}\label{finalest3}
	\left|\min_{\overline\Theta}\pmin^2|\nu_{\ep,\Theta}|-|\pmin^2\nu_{\ep,\Theta}| \right||\log\ep|\leq C |\log\ep|^{-n}.
	\end{equation}
	Finally, by choosing
	$$
	\delta=\delta(\ep)=\min\left(|\log\ep|^{-3(m+n)},|\log\ep|^{-\frac{n+N+1}{\alpha}}, |\log\ep|^{-(\kappa+1)} \right),
	$$
	by combining \eqref{finalest1}, \eqref{finalest2}, and \eqref{finalest3}, we find
	\begin{multline*}
		\int_{S_{\nu_\ep}}\pmin^2|\nabla_{A_\ep}u_\ep|^2+\frac{\pmin^4}{2\ep^2}(1-|u_\ep|^2)^2+|\curl A_\ep|^2\\
		\geq
		\left(|\pmin^2\nu_{\ep,\Theta}|+\sum_{\CC_l\in \GG_0}|\pmin^2\nu_{\ep,\CC_l}|\right)\left(|\log\ep|-C\log |\log\ep|\right)-\frac{C}{|\log \ep |^{n}}\\
		=|\pmin^2\nu_\ep|\left(|\log\ep|-C\log |\log\ep|\right)-\frac{C}{|\log \ep |^{n}}.
	\end{multline*}
    This concludes the proof of the lower bound.
\end{proof}

\begin{remark}\label{remark: hyp a_ep not satisfied}
We note that if the hypotheses on $a_\ep$ in Theorem~\ref{epleveltools} are not satisfied, then we obtain the lower bound
\begin{multline*}
		\int_{S_{\nu_\ep}}\pmin^2|\nabla_{A_\ep}u_\ep|^2+\frac{\pmin^4}{2\ep^2}(1-|u_\ep|^2)^2+|\curl A_\ep|^2\\
		\geq\Lambda_\ep(\nu_\ep,\pmin)\left(|\log\ep|-C\log |\log\ep|\right)-\frac{C}{|\log \ep |^{n}},
	\end{multline*}
where $\Lambda_\ep(\nu_\ep,\pmin)$ is defined in \eqref{Lambdaep}. By replacing \eqref{weighted lower bound} with this inequality, and performing a similar analysis to that of the article, one obtains a lower bound for the first critical field in terms of a modified (and significantly more intricate) version of the weighted isoflux problem, which, in particular, depends on the grid used in the proof of the $\ep$-level estimates.
\end{remark}

\bibliography{references3DPinning}
\end{document}